\documentclass[12pt]{article}

\usepackage{amsfonts}
\usepackage{amssymb}
\usepackage{amscd}
\usepackage{amsmath}
\usepackage[T1]{fontenc}
\usepackage{eurosym}

\newcommand{\secfont}{\sc}
\renewcommand{\section}{\refstepcounter{section}\vskip0.5em\goodbreak%
\secdef\mysecA\mysecB}
\newcommand{\mysecA}[2][default]{\begin{center}{\secfont \thesection.~#2}
\end{center}
\addcontentsline{toc}{subsection}{\protect%
      \numberline{\thesection. }#2}
}
\newcommand{\mysecB}[1]{\addtocounter{section}{-1}%
\begin{center}{\secfont #1}\end{center}\vskip1em\nobreak}

\DeclareFontEncoding{FMS}{}{}
\DeclareFontSubstitution{FMS}{futm}{m}{n}%
\DeclareMathAlphabet{\fcal}{FMS}{futm}{m}{n}

\newcommand{\thmfont}{\sc}
\newtheorem{theorem}[subsection]{\thmfont Theorem}

\newtheorem{definition}[subsection]{\thmfont Definition}
\newtheorem{lemma}[subsection]{\thmfont Lemma}
\newtheorem{proposition}[subsection]{\thmfont Proposition}

\newtheorem{PreExamples}[subsection]{\thmfont Examples}
\newtheorem{PreExample}[subsection]{\thmfon Example}
\newtheorem{PreRemarks}[subsection]{\thmfont Remarks}
\newtheorem{PreRemark}[subsection]{\thmfont Remark}
{\begin{PreExamples}\rm}{\end{PreExamples}}
{\begin{PreExample}\rm}{\end{PreExample}}
{\begin{PreRemarks}\rm}{\end{PreRemarks}}
\newenvironment{remark}%
{\begin{PreRemark}\rm}{\end{PreRemark}}

\newtheorem{PreState}[subsection]{\thmfont\statetype}
{\def\statetype{#1} \begin{PreState}\rm}{\end{PreState}}

\newcommand{\carre}{{\vrule height5pt width5pt depth0pt}}
\newenvironment{proof}%
{\noindent{\it Proof} --\nobreak}{\hfill\carre\vskip1em}
\newenvironment{proofof}[1]%
{\noindent{\it #1} -- \nobreak}{\hfill\carre\vskip1em}

\newcommand{\highlightfont}{\em}
\newcommand{\highlight}[1]{{\highlightfont #1}}

\newcommand{\transpose}[1]{{\vphantom{#1}}^{\mathit t}{#1}}
\def\Id{\mathop{\rm Id{}}\nolimits}

\newcommand{\tildeC}{\widetilde{\mathcal{C}}}
\newcommand{\tildeF}{\widetilde{\mathcal{F}}}
\newcommand{\tildesigma}{\widetilde{\sigma}}
\newcommand{\tildetau}{\widetilde{\tau}}
\newcommand{\barC}{\overline{\mathcal{C}}}

\newcommand{\mathbfit}[1]{\textbf{\textit{#1}}}
\newcommand{\mathbfsc}[1]{\textbf{\textsc{#1}}}
\newcommand{\pfont}[1]{\textbf{\textsf{\textit{#1}}}}

\newcommand{\pS}{\pfont{S}}
\newcommand{\pT}{\pfont{T}}
\newcommand{\pA}{\pfont{A}}
\newcommand{\pM}{\pfont{M}}

\begin{document}

\title{On compositions associated to Frobenius parabolic and seaweed subalgebras of $\mathrm{sl}_{n}(\Bbbk )$}
\author{Michel Duflo and Rupert W.T. Yu\footnote{Research supported by the ANR Project GERCHER Grant 
number ANR-2010-BLAN-110-02.}}

\maketitle

\begin{abstract}
By using a free monoid of operators on the set of compositions (resp. pairs of compositions), 
we establish in this paper a bijective correspondence between Frobenius standard parabolic (resp. seaweed) subalgebras 
and certain elements of this monoid. We prove via this correspondence a conjecture of one of the authors on the number of 
Frobenius standard parabolic  (resp. seaweed) subalgebras of $\mathrm{sl}_{n}(\Bbbk )$ associated to compositions 
(resp. pairs of compositions) with $n-t$ parts (resp. parts in total).
\end{abstract}

\section{Introduction}

Throughout this paper, $\Bbbk$ is an algebraically closed field of characteristic zero.

Let  $n\in \mathbb{N}^{*}$, and let us denote by
$$
\mathcal{C}_{n} = \{ (a_{1},\dots ,a_{r}) \ ; \ r \in \mathbb{N}^{*} \ , \ a_{1} , \dots ,a_{r} \in \mathbb{N}^{*} \hbox{ and } 
a_{1}+ \cdots + a_{r}=n \}
$$
the set of compositions of $n$.

Seaweed subalgebras of $\mathrm{sl}_{n}(\Bbbk)$ were introduced by Dergachev and Kirillov
\cite{DK} to study the index of parabolic subalgebras of $\mathrm{sl}_{n}(\Bbbk)$ \cite{E1}. Like parabolic
subalgebras which can be put in a standard form as block upper triangular matrices of trace zero parametrized
by compositions of $n$, seaweed subalgebras can be put in a standard form parametrized by pairs of 
compositions of $n$. More precisely, the standard seaweed subalgebra of $\mathrm{sl}_{n}(\Bbbk)$ associated 
to the pair $\mathbfsc{a} = (\mathbf{a}^{+}, \mathbf{a}^{-}) \in \mathcal{C}_{n}\times \mathcal{C}_{n}$ is 
$$
\mathfrak{p}_{\mathbfsc{a}} = \mathfrak{p}_{\mathbf{a}^{+}} \cap \transpose{\mathfrak{p}_{\mathbf{a}^{-}}}
$$
where $\mathfrak{p}_{\mathbf{a}^{\pm}}$ denotes the set of block upper triangular matrices of trace zero associated
to $\mathfrak{a}^{\pm}$, and $\transpose{\mathfrak{p}_{\mathbf{a}^{-}}}$ the set of transpose of elements of
$\mathfrak{p}_{\mathbf{a}^{-}}$. For example, for $n=7$, $\mathbf{a}^{+}=(2,3,2)$ and $\mathbf{a}^{-} = (4,3)$, we
have 
$$
\mathfrak{p}_{\mathbf{a}^{+}}= 
\hbox{\scriptsize$
\setlength{\arraycolsep}{1pt}
\renewcommand{\arraystretch}{0.75}
\left(
\begin{array}{ccccccc}
*& * & * & * & * & * & * \\
* & * & * & * & * & * & * \\
0 & 0 & * & * & * & * & * \\
0 & 0 & * & * & * & * & * \\
0 & 0 & * & * & * & * & * \\
0 & 0 & 0 & 0 & 0 & * & * \\
0 & 0 & 0 & 0 & 0 & * & * 
\end{array}
\right)$
}
\cap \mathrm{sl}_{n}(\Bbbk )
\ , \
\mathfrak{p}_{\mathbf{a}^{-}}= 
\hbox{\scriptsize$
\setlength{\arraycolsep}{1pt}
\renewcommand{\arraystretch}{0.75}
\left(
\begin{array}{ccccccc}
*& * & * & * & * & * & * \\
* & * & * & * & * & * & * \\
* & * & * & * & * & * & * \\
* & * & * & * & * & * & * \\
0 & 0 & 0 & 0 & * & * & * \\
0 & 0 & 0 & 0 & * & * & * \\
0 & 0 & 0 & 0 & * & * & * 
\end{array}
\right)$
}
\cap \mathrm{sl}_{n}(\Bbbk )
\ \hbox{ and } \
\mathfrak{p}_{\mathbfsc{a}}= 
\hbox{\scriptsize$
\setlength{\arraycolsep}{1pt}
\renewcommand{\arraystretch}{0.75}
\left(
\begin{array}{ccccccc}
*& * & * & * & 0 & 0 & 0 \\
* & * & * & * & 0 & 0 & 0 \\
0 & 0 & * & * & 0 & 0 & 0 \\
0 & 0 & * & * & 0 & 0 & 0 \\
0 & 0 & * & * & * & * & * \\
0 & 0 & 0 & 0 & 0 & * & * \\
0 & 0 & 0 & 0 & 0 & * & * 
\end{array}
\right)$
}
\cap \mathrm{sl}_{n}(\Bbbk ).
$$
Observe that $\mathfrak{p}_{(\mathbf{a},(n))} = \mathfrak{p}_{\mathbf{a}}$, so 
standard parabolic subalgebras are standard seaweed subalgebras.

Recall from \cite{TYbook} that the index of a Lie algebra $\mathfrak{g}$ is the integer
$$
\mathrm{ind}(\mathfrak{g}) = \min \{ \dim \mathfrak{g}^{f} \, ;\,  f\in \mathfrak{g}^{*}\}
$$ 
where $\mathfrak{g}^{f}$ denotes the annihilator of $f$ in $\mathfrak{g}$ under the coadjoint action.
In \cite{DK}, the authors associate to a seaweed subalgebra $\mathfrak{p}_{\mathbfsc{a}}$, 
a graph called {\it meander graph},  and expressed the index of $\mathfrak{p}_{\mathbfsc{a}}$
in terms of the type of connected components of this graph (which are cycles, segments and points).
When the meander graph of $\mathfrak{p}_{\mathbfsc{a}}$ is a single segment, its index is zero.
Consequently, we obtain many new examples of index zero Lie algebras, also called Frobenius Lie algebras, 
which are objects of interest in representation theory (see for example \cite{E2, O}).

For $\mathbf{a}=(a_{1},\dots ,a_{r}) \in \mathcal{C}_{n}$, let us denote by $p (\mathbf{a}) = r$
the number of parts of the composition $\mathbf{a}$.  The first author computed the number $F_{n,p}$ of 
Frobenius standard parabolic subalgebras $\mathfrak{p}_{\mathbf{a}}$ in $\mathrm{sl}_{n}(\Bbbk )$ with a fixed 
number of parts $p$ up to $n=25$, and he observed that $F_{n,n-t}$ appears to have a nice asymptotic behaviour 
in $n$ for a fixed parity of $n$. He conjectured that $F_{2n,2n-t}$ and $F_{2n+1,2n+1-t}$ are polynomial
in $n$ for large $n$. 

Similar computations for standard seaweed subalgebras in $\mathrm{sl}_{n}(\Bbbk )$ lead to a similar conjecture (without
the condition on the parity of $n$)  on the number $\widetilde{F}_{n,p}$ of Frobenius standard seaweed subalgebras 
$\mathfrak{p}_{\mathbfsc{a}}$ in $\mathrm{sl}_{n}(\Bbbk )$ such that $\mathbfsc{a}=(\mathbf{a}^{+},\mathbf{a}^{-})$ 
and $p(\mathbf{a}^{+}) + p(\mathbf{a}^{-} )=p$. 

In this paper, we prove these conjectures and provide the degree of these polynomials. The main results are:

\begin{theorem}
Let $\varepsilon \in \{ 0 , 1\}$ and $t\in \mathbb{N}$. 
\begin{itemize}
\item[a)] There exists $P_{\varepsilon, t} \in \mathbb{Q}[T]$ of degree $\left[ \frac{t}{2} \right]$ with positive dominant
coefficient such that 
$$
F_{2n+\varepsilon , n+1 -t} = P_{\varepsilon,t} (n)
$$
for large $n$.
\item[b)] There exists $P_{t}  \in \mathbb{Q}[T]$ of degree $\left[ \frac{t}{2} \right]$ 
with positive dominant coefficient such that 
$$
\widetilde{F}_{n , n+1 -t} = P_{t} (n)
$$
for large $n$.
\end{itemize}
\end{theorem}

The main ingredients in proving these results are certain ``index preserving'' operators on the 
set of all compositions (resp. all pairs of compositions). These are the inverses of the reduction 
operations obtained from the inductive formulas of Panyushev in \cite[Theorem 4.3]{Pan}. The 
monoid generated by these operators is free, and we obtain a bijection between the set of 
compositions (resp. pairs of compositions) corresponding to Frobenius parabolic (resp.
seaweed) subalgebras and a subset of this monoid. The determination of $F_{2n+\varepsilon , n+1 -t}$
and $\widetilde{F}_{n , n+1 -t}$ then becomes a word problem in this monoid, which is solved in a purely 
combinatorial way. 

These operators are known to Elashvili \cite{E3} who used them to compute the number of Frobenius
standard seaweed subalgebras in $\mathrm{sl}_{n}(\Bbbk )$. In \cite{CMW}, the authors used a much larger family 
of index preserving operators (on meander graphs) to generate all seaweed subalgebras.  
The operators used here optimize in a certain way the generation of Frobenius parabolic and seaweed 
subalgebras. Note that it is possible to generate all seaweed subalgebras by adding left concatenation operators 
to these operators.

The paper is organized as follows. In Section 2, we define the index preserving operators for pairs
of compositions. We prove that the monoid they generate is free, and establish the bijection between 
the set of pairs of compositions corresponding to Frobenius seaweed subalgebras and a subset of this
monoid. Section 3 investigates certain combinatorial properties of elements of the monoid. These properties
are used in Section 4 to determine $\widetilde{F}_{n,n+1-t}$. Section 5 treats the case of Frobenius
parabolic subalgebras.

We shall conserve the notations introduced in this section throughout this paper. 
\paragraph{Acknowledgements.}
The authors would like to thank Alexander Elashvili for many interesting discussions on Frobenius seaweed
subalgebras.

\section{Index preserving operators on the set of pairs of compositions} 

Set 
$$
\mathcal{C} = \bigcup_{n\in\mathbb{N}^{*}} \mathcal{C}_{n}  \ , \
\tildeC_{n} = \mathcal{C}_{n} \times \mathcal{C}_{n}  \ \hbox{ and } \
\tildeC = \bigcup_{n\in \mathbb{N}^{*}} \tildeC_{n}.
$$
For $\mathbf{a}=(a_{1},\dots ,a_{r}) \in \mathcal{C}$, we set 
$$
p (\mathbf{a}) = r \ , \ s (\mathbf{a}) = a_{1}+\cdots + a_{r} = n 
$$
respectively the number of parts and the sum of the composition.

For $\mathbfsc{a}\in \tildeC_{n}$, we shall write $\mathbfsc{a} = (\mathbf{a}^{+},\mathbf{a}^{-})$ where 
$$
\mathbf{a}^{+} = (a_{1}^{+},\dots ,a_{p(\mathbf{a}^{+})}^{+})
\ \hbox{ and } \ 
\mathbf{a}^{-} = (a_{1}^{-},\dots ,a_{p(\mathbf{a}^{-})}^{-}) 
\in \mathcal{C}_{n}.
$$
We set 
$$
\mathbfit{p} (\mathbfsc{a}) = \bigl( p(\mathbf{a}^{+}) , p(\mathbf{a}^{-}) \bigr) \ , \
p(\mathbfsc{a}) = p(\mathbf{a}^{+})+p(\mathbf{a}^{-}) \ , \
s(\mathbfsc{a}) = s(\mathbf{a}^{\pm}) = n.
$$

For convenience, we shall add a \highlight{null} element denoted by $\mathbfsc{o}$, to obtain 
$\barC = \tildeC \cup \{ \mathbfsc{o} \}$. 
By convention, we set $\mathbfit{p}(\mathbfsc{o}) = (0,0)$ and $s(\mathbfsc{o})=0=p(\mathbfsc{o})$.

Let $m\in\mathbb{N}$.
We define operators $\rho$, $\sigma^{\pm}_{m}$ and $\tau_{m}^{\pm}$ on $\barC$ as follows.
First,
$$
\rho (\mathbfsc{a}) =\left\{
\begin{array}{ll} 
(\mathbf{a}^{-} , \mathbf{a}^{+}) &  \hbox{if } \mathbfsc{a} \in\tildeC, \\
\mathbfsc{o} & \hbox{if }  \mathbfsc{a} =\mathbfsc{o}.
\end{array}
\right.
$$
Next, 
$$
\sigma_{m}^{+} (\mathbfsc{a}) = \left\{
\begin{array}{ll}
\mathbfsc{b} = (\mathbf{b}^{+},\mathbf{b}^{-}) &  \hbox{if } \mathbfsc{a} \in\tildeC, \\
\mathbfsc{o} & \hbox{if }  \mathbfsc{a} =\mathbfsc{o},
\end{array}
\right.
$$
where
$$
\mathbf{b}^{+} = \bigl( (m+2) a_{1}^{+} , a_{2}^{+} , \dots , a_{p(\mathbf{a}^{+})}^{+} \bigr) \ \hbox{ and } \  
\mathbf{b}^{-} = \bigl( (m+1)a_{1}^{+}, a_{1}^{-},\dots ,a_{p(\mathbf{a}^{-})}^{-} \bigr).
$$
Thirdly,
$$
\tau_{m}^{+} (\mathbfsc{a}) = \left\{
\begin{array}{ll}
\mathbfsc{b} = (\mathbf{b}^{+},\mathbf{b}^{-}) &  \hbox{if } \mathbfsc{a} \in\tildeC \hbox{ and } p(\mathbf{a}^{+}) > 1, \\
\mathbfsc{o} & \hbox{if }  \mathbfsc{a}=\mathbfsc{o} \hbox{ or } p(\mathbf{a}^{+})=1,
\end{array}
\right.
$$
where
$$
\mathbf{b}^{+} = \bigl( (m+1) a_{1}^{+} + (m+2) a_{2}^{+} , a_{3}^{+} , \dots , a_{p(\mathbf{a}^{+})}^{+} \bigr)  \ \hbox{ and } \  
\mathbf{b}^{-} = \bigl( m a_{1}^{+} + (m+1)a_{2}^{+} , a_{1}^{-},\dots ,a_{p(\mathbf{a}^{-})}^{-} \bigr).
$$
Finally, we define the compositions
$$
\sigma_{m}^{-} = \rho \, \sigma_{m}^{+} \, \rho \ \hbox{ and } \
\tau_{m}^{-} = \rho \, \tau_{m}^{+} \, \rho
$$
which are analogues by exchanging the roles of the two compositions.

We observe immediately from the definitions the following equalities for $\mathbfsc{a}\in \tildeC$.
\begin{itemize}
\item[{\bf (Ob1)}] We have
$$
\begin{array}{c}
\mathbfit{p}  \bigl( \rho (\mathbfsc{a} ) \bigr) = \mathbfit{p}(\mathbfsc{a}) \ , \ 
\mathbfit{p}  \bigl(  \sigma_{m}^{+} ( \mathbfsc{a}) \bigr) = \mathbfit{p}(\mathbfsc{a}) + ( 0 , 1 )  \ , \
\mathbfit{p}  \bigl(  \sigma_{m}^{-} ( \mathbfsc{a}) \bigr) = \mathbfit{p}(\mathbfsc{a}) + ( 1 ,0 )    \ , \\ [0.3em]
\mathbfit{p}  \bigl( \tau_{m}^{+}  (\mathbfsc{a})  \bigr) =   \left\{
\begin{array}{ll}
\mathbfit{p}(\mathbfsc{a}) + ( -1 , 1 ) & \hbox{if } p(\mathbf{a}^{+}) > 1, \\
(0,0) & \hbox{if } p(\mathbf{a}^{+}) = 1 , 
\end{array}\right. \\ [1em]
\mathbfit{p}  \bigl( \tau_{m}^{-}  (\mathbfsc{a})  \bigr) =   \left\{
\begin{array}{ll}
\mathbfit{p}(\mathbfsc{a}) + ( 1 , -1 ) & \hbox{if } p(\mathbf{a}^{-}) > 1, \\
(0,0) & \hbox{if } p(\mathbf{a}^{-}) = 1.
\end{array}\right.
\end{array}
$$
In particular, 
$$
p \bigl( \rho (\mathbfsc{a} ) \bigr) = p (\mathbfsc{a})  \ , \
p \bigl(  \sigma_{m}^{\pm } ( \mathbfsc{a}) \bigr) = p  (\mathbfsc{a}) + 1\ , \
p \bigl( \tau_{m}^{\pm }  (\mathbfsc{a})  \bigr) = p (\mathbfsc{a}) .
$$
\item[{\bf (Ob2)}] We have 
$$
s \bigl( \rho ( \mathbfsc{a}) \bigr)= s (\mathbfsc{a} ) \ , \
s \bigl( \sigma_{m}^{\pm} ( \mathbfsc{a}) \bigr)= s (\mathbfsc{a} ) + (m+1)a_{1}^{\pm} \ , \
s \bigl( \tau_{m}^{\pm} ( \mathbfsc{a}) \bigr)= s (\mathbfsc{a} ) + ma_{1}^{\pm} + (m+1)a_{2}^{\pm}.
$$
\end{itemize}

Let us denote by $\fcal{S}^{\pm}$ the set of $\sigma_{m}^{\pm}$ where $m\in\mathbb{N}$, 
$\fcal{T}^{\pm}$ the set of $\tau_{m}^{\pm}$ where $m\in\mathbb{N}$,
and 
$$
\fcal{S}= \fcal{S}^{+} \cup \fcal{S}^{-} \ , \ \fcal{T}= \fcal{T}^{+} \cup \fcal{T}^{-} \ \hbox{ and } \
\fcal{A} = \fcal{S} \cup \fcal{T}.
$$
We set  $\fcal{M}$ to be the submonoid generated by $\fcal{A}$ in the set of maps from 
$\barC$ to itself.

\begin{definition}
Let $\mathbfsc{a}\in\tildeC$ and $w\in \fcal{M}$. We say that $w$ is \highlight{$\mathbfsc{a}$-null} if
$w(\mathbfsc{a})=\mathbfsc{o}$. We denote by
$$
\fcal{M}_{\mathbfsc{a}} = \{ w \in \fcal{M}  \ ; \ w (\mathbfsc{a}) \neq \mathbfsc{o} \}
$$ 
the set of elements of $\fcal{M}$ which are not $\mathbfsc{a}$-null.
\end{definition}

From the definitions of the operators, we obtain that if $w=w'w'' \in \fcal{M}_{\mathbfsc{a}}$, then
$w'' \in \fcal{M}_{\mathbfsc{a}}$ and $w' \in \fcal{M}_{w''(\mathbfsc{a})}$.

\begin{proposition}\label{indexpreserving}
Let $\mathbfsc{a}\in\tildeC$ and $w\in \fcal{M}_{\mathbfsc{a}}$. Then the index of the standard
seaweed subalgebras associated to $\mathbfsc{a}$ and $w(\mathbfsc{a})$ are the same.
\end{proposition}
\begin{proof}
It suffices to prove the proposition for $w \in \fcal{M}_{\mathbfsc{a}} \cap \fcal{A}$.

Recall the following result from \cite[Theorem 4.3]{Pan}: let $\mathbfsc{b}\in \tildeC$ 
be such that $b_{1}^{+} < b_{1}^{-}$, and $m,r\in\mathbb{N}$ be the unique integers such that 
$b_{1}^{+} = ( b_{1}^{-}-b_{1}^{+} )m + r$ and $0<r\leqslant b_{1}^{-}-b_{1}^{+}$. In particular, 
$m$ is the unique integer such that 
$$
\frac{m}{m+1}b_{1}^{-}  < b_{1}^{+} \leqslant \frac{m+1}{m+2} b_{1}^{-}.
$$
Then 
$
\mathrm{ind}( \mathfrak{p}_{\mathbfsc{b}} )
= \mathrm{ind}( \mathfrak{p}_{\mathbfsc{c}} )
$
where 
$$
\mathbf{c}^{+} = \bigl( b_{2}^{+} , \dots , b_{p(\mathbf{b}^{+})}^{+}  \bigr)
\ , \ 
\mathbf{c}^{-} = \bigl( (m+1)b_{1}^{-}  - (m+2) b_{1}^{+} , (m+1)b_{1}^{+} - m b_{1}^{-} , b_{2}^{-}, \dots ,
b_{p(\mathbf{b}^{-})}^{-}
\bigr)
$$
where the first entry in $\mathbf{c}^{-}$ is void if $(m+1)b_{1}^{-}  = (m+2) b_{1}^{+}$.

The operators $\sigma_{m}^{-}$ and $\tau_{m}^{-}$ are just the inverse of this operation 
according to the case whether we have $(m+1) b_{1}^{-} = (m+2)b_{1}^{+}$ or not. 

The operators $\sigma_{m}^{+}$ and $\tau_{m}^{+}$ correspond to the case where 
$b_{1}^{+} > b_{1}^{-}$.
\end{proof}

\begin{lemma}\label{uniqueness}
Let $\mathbfsc{a},\mathbfsc{b}\in \tildeC$ and $\eta,\eta' \in \fcal{A}$ be
such that $\eta (\mathbfsc{a}) =\eta'(\mathbfsc{b}) \neq \mathbfsc{o}$. Then $\eta=\eta'$ and $\mathbfsc{a}=\mathbfsc{b}$.
\end{lemma}
\begin{proof}
Let $\mathbfsc{c} = \eta (\mathbfsc{a}) = \eta'(\mathbfsc{b})$.

First, observe that if $\eta \in \fcal{S}^{+} \cup \fcal{T}^{+}$, then $c_{1}^{+} > c_{1}^{-}$, while
if $\eta' \in \fcal{S}^{-} \cup \fcal{T}^{-}$, then $c_{1}^{+} < c_{1}^{-}$. So the only possibilities
for $(\eta ,\eta')$ are 
$$
(\sigma_{p}^{+}, \tau_{q}^{+})  \ ,Ê\  (\tau_{p}^{+}, \sigma_{q}^{+}) \ , \ (\sigma_{p}^{+}, \sigma_{q}^{+}) \ , \
(\tau_{p}^{+}, \tau_{q}^{+}) 
$$
with $p,q\in\mathbb{N}$, and the same pairs with superscript $-$ instead of $+$. By symmetry, it suffices to prove
the lemma for the superscript $+$. 

Suppose that $(\eta ,\eta') = (\sigma_{p}^{+}, \tau_{q}^{+})$. Then  $p(\mathbf{b}^{+}) \geqslant 2$, and 
$$
c_{1}^{+} = (p+2)a_{1}^{+} = (q+1)b_{1}^{+} + (q+2)b_{2}^{+} \ , \
c_{1}^{-} = (p+1)a_{1}^{+} = q b_{1}^{+} + (q+1)b_{2}^{+}.
$$
By looking at the difference, we deduce that $a_{1}^{+} = b_{1}^{+} + b_{2}^{+}$. The case
$(\eta ,\eta') =(\tau_{p}^{+}, \sigma_{q}^{+})$ is analogue.

On the other hand, these equalities give
$$
\left(\begin{array}{cc}
q+1 & q+2 \\
q & q+1
\end{array}
\right)
\left(\begin{array}{c}
b_{1}^{+} \\
b_{2}^{+}
\end{array}
\right)
=
\left(\begin{array}{c}
(p+2)a_{1}^{+} \\
(p+1)a_{1}^{+}
\end{array}
\right)
\hbox{ and } 
\left(\begin{array}{cc}
q+1 & q+2 \\
q & q+1
\end{array}
\right)
\in \mathrm{SL}_{2}(\mathbb{Z}).
$$
So $a_{1}^{+}$ divides both $b_{1}^{+}$ and $b_{2}^{+}$ which is absurd because 
$a_{1}^{+} = b_{1}^{+} + b_{2}^{+}$.

Suppose that $(\eta ,\eta') =(\sigma_{p}^{+}, \sigma_{q}^{+})$. Then $p(\mathbf{a}^{\pm}) = p(\mathbf{b}^{\pm})$, and 
we have $a_{i}^{\varepsilon} = b_{i}^{\varepsilon}$ for all $(i,\varepsilon)\neq (1,+)$,
$$
(p+2)a_{1}^{+} = (q+2)b_{1}^{+}  \ , \ (p+1)a_{1}^{+} = (q+1)b_{1}^{+}.
$$
Again by considering the difference, we obtain that $a_{1}^{+}=b_{1}^{+}$, and hence $p=q$.

Suppose that $(\eta ,\eta') =(\tau_{p}^{+}, \tau_{q}^{+})$. Then $p(\mathbf{a}^{\pm}) = p(\mathbf{b}^{\pm})$, and 
we have $a_{i}^{\varepsilon} = b_{i}^{\varepsilon}$ for all $(i,\varepsilon)\not \in \{ (1,+), (2,+)\}$,
$$
(p+1)a_{1}^{+} + (p+2) a_{2}^{+} = (q+1)b_{1}^{+} + (q+2)b_{2}^{+} \ , \
p a_{1}^{+} + (p+1)a_{2}^{+} = q b_{1}^{+} + (q+1)b_{2}^{+}.
$$ 
The difference gives $a_{1}^{+} + a_{2}^{+} = b_{1}^{+} + b_{2}^{+}= m$, hence
$$
p m + a_{2}^{+} = p a_{1}^{+} + (p+1)a_{2}^{+} = q b_{1}^{+} + (q+1)b_{2}^{+} = qm+b_{2}^{+}.
$$
By the uniqueness of Euclidean division, we have $p=q$ and $a_{2}^{+} = b_{2}^{+}$, which 
in turn implies that $a_{1}^{+} = b_{1}^{+}$. So the proof is complete.
\end{proof}

\begin{proposition}\label{freemonoid}
\begin{itemize}
\item[a)] The monoid $\fcal{M}$ is free in the alphabet $\fcal{A}$. 
\item[b)] For $\mathbfsc{a}\in \tildeC$, the map 
$$
\mathrm{ev}_{\mathbfsc{a}} : \fcal{M}_{\mathbfsc{a}} \longrightarrow \tildeC \ , \ w \mapsto w(\mathbfsc{a}) \ ,
$$
is injective.
\item[c)] Let $\mathbfsc{a}, \mathbfsc{b} \in \tildeC$ be such that $\mathbfsc{b}\not\in 
\mathrm{im} (\mathrm{ev}_{\mathbfsc{a}})$.  Then $\mathrm{im}( \mathrm{ev}_{\mathbfsc{a}}) \cap 
\mathrm{im} (\mathrm{ev}_{\mathbfsc{b}})= \emptyset$.  
\end{itemize}
\end{proposition}
\begin{proof}
These are direct consequences of Lemma \ref{uniqueness}.
\end{proof}

\begin{remark}\label{meander}
It is a basic meander graph observation that if $\mathfrak{p}_{\mathbfsc{a}}$ is Frobenius, then
$a_{1}^{+}=a_{1}^{-}$ if and only if $\mathbfsc{a}=\bigl( (1),(1) \bigr)$. We can also recover this
from \cite[Theorem 4.2]{Pan} or \cite{TY, TYbook}.
\end{remark}

Let us denote by $\tildeF$ the set of elements of $\tildeC$ corresponding to Frobenius standard seaweed
subalgebras. Let $\mathbfsc{i} = \bigl( (1),(1) \bigr) \in \tildeF$.

\begin{theorem}\label{frobeniusbijection}
The map $\mathrm{ev}_{\mathbfsc{i}}$ defines a bijection from $\fcal{M}_{\mathbfsc{i}}$
to $\tildeF$.
\end{theorem}
\begin{proof}
By Proposition \ref{indexpreserving}, we have $\mathrm{im}( \mathrm{ev}_{\mathbfsc{i}})\subset \tildeF$, 
and so the map $\mathrm{ev}_{\mathbfsc{i}}$ is injective by Proposition \ref{freemonoid}.

Let $\mathbfsc{a}\in \tildeF\setminus \{\mathbfsc{i}\}$. Then $a_{1}^{+}\neq a_{1}^{-}$ by Remark \ref{meander}. 
By the result of Panyushev stated in the proof of Proposition \ref{indexpreserving}, there exists $\mathbfsc{b}\in \tildeF$
and $\eta \in \fcal{A}$ such that $\eta (\mathbfsc{b}) = \mathbfsc{a}$. Since $s( \mathbfsc{b} ) < s (\mathbfsc{a})$,
by iterating the process, we arrive at $\mathbfsc{i}$, and the result follows.
\end{proof}

We shall denote by $\iota$ the identity element in $\fcal{M}$. 

Since $\fcal{M}$ is free in the alphabet $\fcal{A}$, if $w= \eta_{1}\cdots \eta_{k}\in\fcal{M}\setminus \{ \iota\}$  with
$\eta_{1},\dots, \eta_{k}\in\fcal{A}$, then the integers 
$$
\begin{array}{c}
\ell (w) = k \ , \ \sigma^{\pm} (w) = \sharp\{\ i \ ;\  \eta_{i}\in\fcal{S}^{\pm} \} \ , \ 
\sigma(w) =\sigma^{+}(w)+\sigma^{-}(w) \ , \\ [0.3em]
\tau^{\pm}  (w)= \sharp\{ \ i \ ; \ \eta_{i}\in\fcal{T}^{\pm} \} \ , \
\tau (w) =\tau^{+}(w)+\tau^{-}(w) 
\end{array}
$$
are well-defined.  By convention, we define these integers to be zero when $w=\iota$.

\begin{itemize}
\item[{\bf (Ob3)}] It follows immediately from (Ob1) that  for any $\mathbfsc{a}\in\tildeC$ and $w\in \fcal{M}_{\mathbfsc{a}}$, 
we have 
$$
\begin{array}{c}
\ell (w) = \sigma (w)+ \tau (w) \ , \ 
p\bigl(w(\mathbfsc{a}) \bigr) =  p(\mathbfsc{a}) + \sigma (w) \ ,  \\ [0.3em]
\mathbfit{p} \bigl(w(\mathbfsc{a}) \bigr) =  \mathbfit{p}(\mathbfsc{a}) 
+ \bigl( \sigma^{-}(w)+\tau^{-}(w) - \tau^{+}(w) \  , \  \sigma^{+}(w)+\tau^{+}(w) - \tau^{-}(w)\bigr).
\end{array}
$$ 
\end{itemize}

\section{Properties of these operators relative to $w$-sequences}

Let $\mathbfsc{a}\in \tildeC$ and $w = \eta_{1}\cdots \eta_{k} \in \fcal{M}_{\mathbfsc{a}}\setminus \{ \iota \}$ with
$\eta_{1},\dots, \eta_{k}\in\fcal{A}$. Set $m_{k} = s\bigl( \eta_{k}(\mathbfsc{a}) \bigr) - s(\mathbfsc{a})$, and
for $1\leqslant i\leqslant k-1$, we set
$$
m_{i} = s\bigl( ( \eta_{i} \cdots \eta_{k} )(\mathbfsc{a}) \bigr) - s\bigl( (\eta_{i+1} \cdots \eta_{k} )(\mathbfsc{a}) \bigr).
$$ 
The sequence $\mathbf{m} = (m_{1},\dots ,m_{k})$ will be called the \highlight{$w$-sequence of $\mathbfsc{a}$}.
It codes the progression of the sum of the composition of the successive application of $\eta_{i}$ to $\mathbfsc{a}$.
By convention, the $\iota$-sequence of $\mathbfsc{a}$ is empty.

\begin{itemize}
\item[{\bf (Ob4)}] For all $1\leqslant i\leqslant k$, we have $m_{i} \geqslant 1$, and 
$$
\ell (w) + \beta (\mathbf{m}) \leqslant m_{1} + \cdots + m_{k} = 
s\bigl( w(\mathbfsc{a}) \bigr) - s (\mathbfsc{a})
$$
where $\beta (\mathbf{m}) = \sharp \{ \ i \ ; \ m_{i} > 1\}$.
\end{itemize}

\begin{proposition}\label{identical}
Let $\mathbfsc{a}, \mathbfsc{b} \in \tildeC$ and $w\in \fcal{M}$
be such that 
\begin{itemize}
\item[i)] $p(\mathbf{a}^{\pm}) \geqslant \ell (w)+1$ and $p(\mathbf{b}^{\pm}) \geqslant \ell (w)+1$, and
\item[ii)] we have 
$a_{i}^{\varepsilon} = b_{i}^{\varepsilon}$ for all $1\leqslant i \leqslant \ell (w)+1$ and $\varepsilon \in \{ +, - \}$.
\end{itemize}
Then $w\in \fcal{M}_{\mathbfsc{a}}\cap \fcal{M}_{\mathbfsc{b}}$, and 
the $w$-sequence of $\mathbfsc{a}$ and the $w$-sequence of $\mathbfsc{b}$ are identical. It is completely
determined by $a_{1}^{\pm}, \dots ,a_{\ell (w) + 1}^{\pm}$. In particular, we have
$$
s\bigl( w(\mathbfsc{a}) \bigr) - s(\mathbfsc{a})=s\bigl( w(\mathbfsc{b}) \bigr) - s(\mathbfsc{b})
\hbox{ and } 
c_{1}^{\varepsilon} = d_{1}^{\varepsilon}
$$
where $\mathbfsc{c} = w(\mathbfsc{a})$, $\mathbfsc{d} = w(\mathbfsc{b})$ and $\varepsilon \in \{ + , - \}$.
\end{proposition}
\begin{proof}
These are direct consequences from the definitions of the operators $\sigma_{m}^{\pm}$ and 
$\tau_{m}^{\pm}$.
\end{proof}

\begin{lemma}\label{(1)-step}
Let $w\in \fcal{M}$, $\eta \in \fcal{A}$, $\mathbfsc{a}\in \tildeC$ and $\mathbfsc{b}=\eta (\mathbfsc{a})$.
\begin{itemize}
\item[a)] If $s\bigl( w(\mathbfsc{a}) \bigr) - s(\mathbfsc{a})=1$, then $w\in \fcal{A}$, and we are in one 
of the following four cases:
\begin{itemize}
\item[i)] $w = \sigma_{0}^{+}$, $a_{1}^{+}=1$ and $b_{1}^{-}=1$.
\item[ii)] $w = \sigma_{0}^{-}$, $a_{1}^{-}=1$ and $b_{1}^{+}=1$.
\item[iii)] $p(\mathbf{a}^{+} )\geqslant 2$, $w = \tau_{0}^{+}$, $a_{2}^{+}=1$ and $b_{1}^{-}=1$.
\item[iv)] $p(\mathbf{a}^{-} )\geqslant 2$, $w = \tau_{0}^{-}$, $a_{2}^{-}=1$ and $b_{1}^{+}=1$.
\end{itemize}
\item[b)] If $\min (b_{1}^{+}, b_{1}^{-})=1$, then $\eta \in\{ \sigma_{0}^{\pm}, \tau_{0}^{\pm}\}$.
Moreover, if $\eta \in \{ \sigma_{0}^{\pm} \}$, then $s\bigl( \eta(\mathbfsc{a}) \bigr) - s(\mathbfsc{a})=1$.
\end{itemize}
\end{lemma}
\begin{proof}
These are straightforward checks from the definitions.
\end{proof}

\begin{lemma}\label{(1,1,1)-step}
Let $\varepsilon_{1},\varepsilon_{2}\in \{+,-\}$ and $\mathbfsc{a}\in \tildeC$ be such that 
the $\sigma_{0}^{\varepsilon_{1}}\sigma_{0}^{\varepsilon_{2}}$-sequence of $\mathbfsc{a}$
is $(1,1)$. Then $\varepsilon_{1}\neq \varepsilon_{2}$ and 
$a_{1}^{\varepsilon_{2}} = 1$.

Furthermore, if $\mathbfsc{a}\in \tildeF\setminus \{\mathbfsc{i}\}$ and $\eta \in \fcal{A}$ is such that the 
$\eta\sigma_{0}^{\varepsilon_{1}}\sigma_{0}^{\varepsilon_{2}}$-sequence of $\mathbfsc{a}$
is $(1,1,1)$, then $\eta = \sigma_{0}^{\varepsilon_{2}}$.
\end{lemma}
\begin{proof}
These are direct consequences of Lemma \ref{(1)-step} because when 
$\mathbfsc{a}\in \tildeF\setminus \{\mathbfsc{i}\}$, we have $a_{1}^{+}\neq a_{1}^{-}$ (Remark \ref{meander}).
\end{proof}

Let us denote 
$$
\fcal{Z} = \{ \sigma_{0}^{\varepsilon_{1}} \cdots \sigma_{0}^{\varepsilon_{r}} \, ; \,
r\in\mathbb{N}^{*}, \ \varepsilon_{1},\dots ,\varepsilon_{r}\in \{ + , - \} \hbox{ and } \varepsilon_{i}\neq \varepsilon_{i+1} 
\hbox{ for all } 1\leqslant i \leqslant r-1 \} \subset \fcal{M}.
$$
If $\sigma_{0}^{\varepsilon_{1}} \cdots \sigma_{0}^{\varepsilon_{r}} \in \fcal{Z}$, then the $\varepsilon_{i}$'s
are completely determined by  $\varepsilon_{r}$. We shall write $\zeta_{r}^{\varepsilon_{r}}$ instead of 
$\sigma_{0}^{\varepsilon_{1}} \cdots \sigma_{0}^{\varepsilon_{r}}$.

\begin{lemma}\label{zeta}
Let $\mathbfsc{a}\in\tildeC$ and $w=\eta_{1}\cdots \eta_{k} \in \fcal{M}_{\mathbfsc{a}}$ be such that
$\eta_{1},\dots ,\eta_{k}\in\fcal{S}$. If the $w$-sequence of $\mathbfsc{a}$ is $(1,\dots ,1)$, then 
$w\in \fcal{Z}$ and we have 
\begin{itemize}
\item[i)] either $a_{1}^{+} = 1$ and $w = \zeta_{k}^{+}$,
\item[ii)] or $a_{1}^{-}=1$ and $w=\zeta_{k}^{-}$.
\end{itemize}
\end{lemma}
\begin{proof}
This is a direct consequence of Lemmas \ref{(1)-step} and \ref{(1,1,1)-step}.
\end{proof}

\begin{remark}\label{zetadescription}
Let $\mathbfsc{a}\in\tildeC$ be such that $a_{1}^{+}=1$. Then for $m\in\mathbb{N}$, we have
$$
\zeta_{2m}^{+} (\mathbfsc{a}) = 
\bigl(
( 1 , \underbrace{2 , \dots ,2}_{m \text{ times}} , a_{2}^{+} , \dots ,a_{p(\mathbf{a}^{+})}^{+} ) ,
( \underbrace{2 , \dots ,2}_{m \text{ times}} , a_{1}^{-} , \dots ,a_{p(\mathbf{a}^{-})}^{-} )
\bigr)
$$
and 
$$
\zeta_{2m+1}^{+} (\mathbfsc{a}) = 
\bigl(
( \underbrace{2 , \dots ,2}_{m+1 \text{ times}} , a_{2}^{+} , \dots ,a_{p(\mathbf{a}^{+})}^{+} ) ,
( 1,\underbrace{2 , \dots ,2}_{m \text{ times}} , a_{1}^{-} , \dots ,a_{p(\mathbf{a}^{-})}^{-} )
\bigr).
$$
We have similar descriptions for $\zeta_{2m}^{-}$ and $\zeta_{2m+1}^{-}$ when $a_{1}^{-}=1$.
\end{remark}

\section{On the growth of $\sharp \tildeF_{n,n+1-t}$ with respect to $n$}

For $p\in\mathbb{N}^{*}$, we set 
$$
\tildeF_{n,p}= \{ \mathbfsc{a} \in \tildeF \ ; \ s(\mathbfsc{a}) = n \hbox{ and } 
p(\mathbfsc{a}) = p \}.
$$

\begin{proposition}\label{empty}
Let $n,p\in\mathbb{N}^{*}$, then 
$\tildeF_{n,p}$ is empty if $p > n+1$.
\end{proposition}
\begin{proof}
Let $p>n+1$ and suppose that $\tildeF_{n,p}$ is non empty. By Theorem \ref{frobeniusbijection},
there exists $w\in \fcal{M}_{\mathbfsc{i}}$ such that $w(\mathbfsc{i}) \in \tildeF_{n,p}$. Since $p>n+1$,
by (Ob3), we have 
$$
\ell (w) \geqslant \sigma (w) = p\bigl( w(\mathbfsc{i}) \bigr) - p(\mathbfsc{i}) >  n. 
$$
But this is impossible because by (Ob4), we have  
$$
n - 1 = s\bigl( w( \mathbfsc{i}) \bigr) - s(\mathbfsc{i}) \geqslant \ell (w)  + \beta (\mathbf{m}) \geqslant \ell (w).
$$
We conclude that $\tildeF_{n,p}$ is empty if $p > n+1$.
\end{proof}

\begin{theorem}\label{growth}
Let $t\in \mathbb{N}$. There exists a polynomial $P_{t} \in\mathbb{Q}[T]$ of degree $\left[ \frac{t}{2}\right]$
and with positive dominant coefficient such that 
$$
\sharp \tildeF_{n,n+1-t} = P_{t}(n)
$$ 
for large $n$.
\end{theorem}

\begin{proof}
Let $n\in\mathbb{N}^{*}$. By Theorem \ref{frobeniusbijection}, $\tildeF_{n,n+1-t}$ is in bijection with 
$$
\fcal{W}_{n} = \{ w\in \fcal{M}_{\mathbfsc{i}} \ ; \ w(\mathbfsc{i}) \in \tildeF_{n,n+1-t} \}.
$$ 

We shall first establish some properties of elements in $\fcal{W}_{n}$.
Let us fix $w=\eta_{1}\cdots \eta_{k}\in\fcal{W}_{n}$ with $\eta_{1},\dots ,\eta_{k}\in\fcal{A}$ and
denote by $\mathbf{m}=(m_{1},\dots ,m_{k})$ the $w$-sequence of $\mathbfsc{i}$.

\paragraph{Conditions on $\ell (w)$, $\sigma (w)$, $\tau (w)$ and $\beta (\mathbf{m})$.}

By (Ob3), we have
$$
n+1-t = p\bigl( w(\mathbfsc{i}) \bigr) = p(\mathbfsc{i}) + \sigma (w) = \sigma(w)+2.
$$
\begin{itemize}
\item[{\bf (Ob5)}] Hence by (Ob4), we obtain the inequalities
$$
n-1-t  = \sigma( w) \leqslant \ell (w) \leqslant \ell (w)  + \beta (\mathbf{m})  \leqslant 
s\bigl( w(\mathbfsc{i}) \bigr) - s(\mathbfsc{i}) = n - 1.
$$
In particular, we have
$$
0 \leqslant \tau (w) + \beta (\mathbf{m}) = \ell (w) + \beta (\mathbf{m}) - \sigma (w)  \leqslant t.
$$
\end{itemize}

\paragraph{Conditions on the $m_{i}$'s and $\eta_{i}$'s.} 
Again, by using (Ob4), we obtain 
$$
(\ell(w)-1) + m_{i} \leqslant m_{1}+\cdots + m_{k} = s\bigl( w(\mathbfsc{i}) \bigr) - s(\mathbfsc{i}) = n - 1.
$$
It follows from (Ob5) that 
$$
1\leqslant m_{i} \leqslant n-\ell( w) \leqslant t+1.
$$
\begin{itemize}
\item[{\bf (Ob6)}] Hence we deduce from (Ob2) that
$$
\eta_{i} \in \fcal{A}_{t} =\{ \sigma_{0}^{\pm} , \dots ,\sigma_{t}^{\pm}, \tau_{0}^{\pm}, \dots ,\tau_{t}^{\pm}\}
$$
for all $1\leqslant i \leqslant k$. 
\end{itemize}
Let us denote by $\fcal{M}_{t}$ the free submonoid of $\fcal{M}$ generated
by $\fcal{A}_{t}$.

\paragraph{A first decomposition of $w$.} Recall that we have
$$
\begin{array}{cccccccc}
w & = & \eta_{1} & \eta_{2} & \cdots & \eta_{k} & \\
\mathbf{m} & = & (m_{1})& (m_{2}) & \cdots & (m_{k}) & 
\end{array}
$$
where the right hand side of the equality in the second line is understood to be the concatenation
of the sequence. We rewrite the above decomposition by grouping pairs $(\eta_{i},m_{i})$ according 
to it being of the form $(\sigma^{\pm}_{0},1)$ or not. It follows that there exists $q\in\mathbb{N}$ such that
$$
\begin{array}{cccccccccc}
w & = & w_{0} & z_{1}  & w_{1} & \cdots & w_{q-1} & z_{q} & w_{q} \\
\mathbf{m} & = & \mathbf{m}_{0}& \mathbf{n}_{1} & \mathbf{m}_{1} & \cdots & \mathbf{m}_{q-1} & 
\mathbf{n}_{q} & \mathbf{m}_{q}
\end{array}
$$
where the $w_{i}$'s does not contain any pairs of the form $(\sigma^{\pm}_{0},1)$ 
and the $z_{i}$'s contains only pairs of the form $(\sigma^{\pm}_{0},1)$. By construction, 
$\mathbf{n}_{i}=(1,\dots ,1)$ is not empty for $1\leqslant i\leqslant q$, so $z_{1},\dots ,z_{q}\in  \fcal{Z}$.
Note also that $w_{1},\dots ,w_{q-1} \in \fcal{M}\setminus \{ \iota\}$.

This decomposition is therefore unique, and we denote $\Delta (w) = (w_{0},\dots ,w_{q})$, 
that we shall call the \highlight{$\Delta$-sequence} of $w$.

\paragraph{Conditions on $q$ and $w_{i}$.}
It follows from Lemma \ref{(1)-step} part b) that for $1\leqslant i\leqslant q-1$, we have 
$w_{i} =\tau_{0}^{\varepsilon} w_{i}'$ for some $w_{i}' \in \fcal{M}$ and $\varepsilon \in \{ + , -\}$. 
Hence 
$$
q-1 \leqslant \tau (w) \leqslant t
$$
by (Ob5).

For $0\leqslant i\leqslant q$, there exist $w',w''\in \fcal{M}$ such that $w=w'w_{i}w''$. We observe
readily that $\mathbf{m}_{i}$ is the $w_{i}$-sequence of $w''(\mathbfsc{a})$.  By construction, if
$(\eta_{j},m_{j})$ is a pair occuring in $w_{i}$, then either $m_{j}>1$ or $\eta_{j} \in \fcal{T}$.
Hence 
$$
\ell (w_{i}) \leqslant \beta (\mathbf{m}_{i}) + \tau (w_{i})  \leqslant t
$$ 
by (Ob5). Consequently, we have
$$
w_{i} \in \fcal{N}_{t} = \{ u \in \fcal{M}_{t} \ ; \ \ell (u) \leqslant t \}.
$$
The set $\fcal{N}_{t}$ is finite and depends only on $t$. We deduce that 
$$
\fcal{W}_{n} = \bigcup_{q=0}^{t+1} \bigcup_{\Delta \in \fcal{N}_{t}^{q+1}} \fcal{W}_{n,\Delta}
$$
where $\fcal{W}_{n,\Delta} = \{ u \in \fcal{W}_{n} \, ; \, \Delta (u) = \Delta \}$. Of course, it may happen that
$\fcal{W}_{n,\Delta}$ is empty.

\paragraph{A second decomposition of $w$.} 
Let  $r\in\mathbb{N}$ be such that $r\leqslant t+1$. Given $\mathbf{w}=(w_{0},\dots ,w_{r}) \in \fcal{N}_{t}^{r+1}$ and 
$\mathbf{u}=(u_{1}, \dots ,u_{r}) \in  \fcal{Z}^{r}$,
we set
$$
x_{\mathbf{w},\mathbf{u}} = w_{0} u_{1} w_{1} \cdots w_{r-1} u_{r} w_{r}.
$$
We say that $x_{\mathbf{w},\mathbf{u}}$ verifies $(*)$  if for $h=1,\dots ,r$, we have
$$
\ell (u_{h})
< 2\ell (w_{0}  u_{1} w_{1} \cdots w_{h -1}) + 1.
$$
In particular, if $x_{\mathbf{w},\mathbf{u}}$ verifies $(*)$, then
$$
\ell (u_{h}) < 2 \sum_{k=0}^{h-1} 3^{h-k-1} \ell (w_{k}) + 3^{h-1}.
$$
It follows from the definition of $\fcal{Z}$ that the set
$$
\fcal{R}_{t} = \{ x_{\mathbf{w},\mathbf{u}} \, ; \,  ( \mathbf{w} , \mathbf{u}) \in \fcal{N}_{t}^{r+1} \times 
\fcal{Z}^{r} \hbox{ for some $r\leqslant t+1$ and } x_{\mathbf{w},\mathbf{u}} \hbox{ verifies $(*)$} \}
$$
is finite and its cardinal depends only on $t$.

We can extract from the first decomposition a new decomposition 
$$
\begin{array}{cccccccccc}
w & = & v_{0} & u_{1} & v_{1} & \cdots & v_{s-1} & u_{s} & v_{s} \\
\mathbf{m} & = & \mathbf{m}_{0}' & \mathbf{n}_{1}' & \mathbf{m}_{1}' & \cdots &
\mathbf{m}_{s-1}' & \mathbf{n}_{s}' & \mathbf{m}_{s}' 
\end{array}
$$
where $v_{0},\dots ,v_{s} \in \fcal{R}_{t}$ and $u_{1},\dots ,u_{s}\in  \fcal{Z}$ verifying
$$
\ell (u_{i}) \geqslant 2 \ell( v_{i-1})+1
$$
for $i=1,\dots ,s$. Moreover, $\mathbf{n}_{i}'=(1,\dots ,1)$ for $1\leqslant i\leqslant s$.

Taking into account of the conditions on the $u_{i}$'s, this decomposition is unique. 
We denote $\Upsilon (w) = (v_{0}, \dots ,v_{s})$, that we shall call the $\Upsilon$-sequence 
of $w$.

\paragraph{Condition on $s$.}
Clearly, we have by construction that
$$
s-1 \leqslant q-1 \leqslant \tau (w).
$$
Moreover, since each $v_{i}$ ends with some $w_{j}$ and $\ell (u_{i+1}) \geqslant 2$ for $i=1,\dots ,s-1$, 
we have by Lemma \ref{(1,1,1)-step} that 
$$
\beta (\mathbf{m}_{i}') > 0 
$$
for $i=1,\dots ,s-1$. Hence $s-1 \leqslant \beta (\mathbf{m})$, and by (Ob5), we obtain that
$$
2(s-1) \leqslant  \tau (w) + \beta (\mathbf{m}) \leqslant t.
$$
Thus
\begin{equation}\label{sbound}
s \leqslant \left[ \frac{t}{2}\right] +1.
\end{equation}

It follows that 
$$
\fcal{W}_{n} = \bigcup_{s=0}^{\left[ \frac{t}{2}\right] +1} \bigcup_{\Upsilon \in \fcal{R}_{t}^{s+1}}
\fcal{W}_{n,\Upsilon}
$$
where $\fcal{W}_{n,\Upsilon} =\{ u\in \fcal{W}_{n} \, ; \, \Upsilon (u) = \Upsilon\}$.
Of course, $\fcal{W}_{n,\Upsilon}$ can be empty.

\paragraph{The set $\fcal{W}_{[\Upsilon^{\circ}]}$.}
For any $w\in\fcal{M}$, $\overline{w} = \rho w \rho$ is the word obtained from $w$ by inversing all the
signs. This defines an involution on $\fcal{M}$, and it follows from the definitions that $\fcal{N}_{t}$ and 
$\fcal{R}_{t}$ are stable under this involution. Observe also that $\ell (\overline{w}) = \ell (w)$.

For $\Upsilon^{\circ} = (v_{0},\dots ,v_{s}) \in \fcal{R}_{t}^{s+1}$, we set
$$
[\Upsilon^{\circ}] = \{ (v_{0}' , \dots , v_{s}') \in \fcal{R}_{t}^{s+1} \, ; \, v_{i}'  = v_{i} \hbox{ or } \overline{v_{i}}\, \hbox{ for all }
i=0,\dots , s\}
$$
and
$$
\fcal{W}_{n, [\Upsilon^{\circ}]} = \bigcup\limits_{\Upsilon \in [\Upsilon^{\circ}]} \fcal{W}_{n, \Upsilon} \ , \
\fcal{W}_{[\Upsilon^{\circ}]} = \bigcup_{n\geqslant 1} \fcal{W}_{n,[\Upsilon^{\circ}]}.
$$
We shall prove that when $\fcal{W}_{[\Upsilon^{\circ}]}$ is non empty, there is a polynomial $P_{[\Upsilon^{\circ}]}\in \mathbb{Q}[T]$ 
of degree $s-1$ and with strictly positive dominant coefficient such that 
$\sharp \fcal{W}_{n,[\Upsilon^{\circ}]} = P_{[\Upsilon^{\circ}]}(n)$ for $n$ sufficiently large.

Let us fix $\Upsilon^{\circ} = (v_{0},\dots ,v_{s})$ and suppose that $\fcal{W}_{[\Upsilon^{\circ}]}$
is not empty.

\paragraph{Elements of $\fcal{W}_{[\Upsilon^{\circ}]}$.}
Let $w\in \fcal{W}_{n,\Upsilon}$ for some $n\geqslant 1$ and $\Upsilon  = (w_{0},\dots ,w_{s}) \in [\Upsilon^{\circ}]$. We have
$$
\begin{array}{cccccccccc}
w & = & w_{0} & u_{1} & w_{1} & \cdots & w_{s-1} & u_{s} & w_{s} & \in \fcal{W}_{n,\Upsilon}\\
\mathbf{m} & = & \mathbf{m}_{0}' & \mathbf{n}_{1}' & \mathbf{m}_{1}' & \cdots &
\mathbf{m}_{s-1}' & \mathbf{n}_{s}' & \mathbf{m}_{s}' 
\end{array}
$$
where $\mathbf{m}$ is the $w$-sequence of $\mathbfsc{i}$, and 
$u_{1},\dots , u_{s} \in  \fcal{Z}$ verify $\ell (u_{i}) \geqslant 2 \ell (w_{i-1})+1 = 2\ell (v_{i-1})+1$ for $1\leqslant i \leqslant s$.

We define 
$$
\varepsilon (w) = (\varepsilon_{1},\dots ,\varepsilon_{s}) \in \{ +,-\}^{s} \ , \
\kappa (w) = ( k_{1},\dots ,k_{s} ) \in \mathbb{N}^{s}
$$
be such that $u_{i} = \zeta_{k_{i}}^{\varepsilon_{i}}$ for all $1\leqslant i\leqslant s$. Thus $k_{i}=\ell (u_{i})$, and we set
$$
\pi (w) = (\pi_{1},\dots ,\pi_{s} ) \in \{ +,-\}^{s}
$$
where for $1\leqslant i\leqslant s$, 
$$
\pi_{i}  = \left\{  
\begin{array}{ll}
+ & \hbox{if $\ell (u_{i})$ is even,}\\
- & \hbox{if $\ell (u_{i})$ is odd.}
\end{array}
\right.
$$
Note that $w$ is completely determined by $\Upsilon$, $\kappa(w)$, $\varepsilon(w)$. Morevoer,  
if $w_{s}\neq \iota$, then $\varepsilon_{s}$ is uniquely determined by $w_{s}$
by Lemma \ref{zeta} and Remark \ref{meander}.

For $1\leqslant i \leqslant s$,
set $\mathbfsc{a}= (w_{i} u_{i+1} w_{i+1} \cdots w_{s-1} u_{s} w_{s}) (\mathbfsc{i})$,
and $\mathbfsc{b} = u_{i}(\mathbfsc{a})$.
Then the $u_{i}$-sequence of $\mathbfsc{a}$ is $\mathbf{n}_{i}' =(1,\dots ,1)$. 
Moreover, by Lemma \ref{zeta}, $a_{1}^{\varepsilon_{i}} =1$.

Let 
$$
\mathbfsc{c}=  \zeta_{2\ell (w_{i-1})+2}^{\pi_{i}\varepsilon_{i}} (\mathbfsc{i})
$$
where $\pi_{i}\varepsilon_{i}$ is the usual multiplication of signs.
Since $\ell (u_{i}) \geqslant 2\ell (w_{i-1})+1$, we deduce from 
Remark \ref{zetadescription} that
$$
b_{j}^{\pm} = c_{j}^{\pm}
$$
for $j=1,\dots , \ell (w_{i-1})+1$. Hence by Proposition \ref{identical}, 
$$
s \bigl( w_{i-1} u_{i}(\mathbfsc{a}) \bigr) - s \bigl( u_{i}(\mathbfsc{a}) \bigr)
= s\bigl( w_{i-1}(\mathbfsc{b}) \bigr) - s (\mathbfsc{b}) =
s \bigl( w_{i-1}  (\mathbfsc{c}) \bigr) - s( \mathbfsc{c})  = \lambda_{i-1}.
$$
Observe that $\lambda_{i-1}$ depends only on $w_{i-1}$, $\pi_{i}$ and $\varepsilon_{i}$, 
and that again by  Proposition \ref{identical}, $\varepsilon_{i-1}$ is completely determined by
$w_{i-1}$, $\pi_{i}$ and $\varepsilon_{i}$. 

Setting $\lambda_{s} = s\bigl( w_{s} (\mathbfsc{i}) \bigr)$, we obtain that
$$
n = s\bigl( w(\mathbfsc{i}) \bigr) = \lambda_{0} + \cdots + \lambda_{s} + \ell (u_{1}) + \cdots + \ell (u_{s}),
$$
and the sum $\lambda_{0}+\cdots +\lambda_{s}$ depends only on 
$\Upsilon$, $\pi (w)$ and $\varepsilon (w)$. 
We shall write 
$$
\lambda_{\Upsilon , \pi(w) , \varepsilon (w)} = \lambda_{0}+\cdots + \lambda_{s}. 
$$
Then
$$
\ell (u_{1}) + \cdots + \ell (u_{s}) = n- \lambda_{\Upsilon, \pi(w), \varepsilon(w)}.
$$

Now, suppose that $k_{1},\dots ,k_{s}$ are positive integers such that 
$k_{1}+\cdots + k_{s} = n- \lambda_{\Upsilon, \pi(w), \varepsilon (w)}$. If for all $1\leqslant i\leqslant s$,
$k_{i} \geqslant 2\ell (w_{i-1})+1$ and the integers $k_{i}$, $\ell (u_{i})$ are of the same parity, then
one verifies easily that 
$$
w' = w_{0}u_{1}' w_{1}\cdots w_{s-1} u_{s}' w_{s} \in \fcal{W}_{n,\Upsilon}
$$ 
where $u_{i}' = \zeta_{k_{i}}^{\varepsilon_{i}}$. In particular, $\pi (w') = \pi (w)$. 

This allows us to construct,  for any large $n$ under some parity conditions, new elements in 
$\fcal{W}_{[\Upsilon^{\circ}]}$ whose $\Upsilon$-sequence
is $\Upsilon$. 

Let $\Xi = [\Upsilon^{\circ} ] \times \{ +, -\}^{s} \times \{ +, -\}^{s}$, then 
we have a finite disjoint union 
$$
\fcal{W}_{[\Upsilon^{\circ}]} = \bigcup_{(\Upsilon , \pi  , \varepsilon) \in \Xi} \fcal{W}_{\Upsilon , \pi, \varepsilon}
$$
where $\fcal{W}_{\Upsilon , \pi, \varepsilon} = \{ w \in \fcal{W}_{[\Upsilon^{\circ}]} \, ; \, 
\Upsilon(w)= \Upsilon, \pi (w) = \pi , \varepsilon (w) = \varepsilon\}$.

In particular, we deduce from the above discussion that if $\fcal{W}_{\Upsilon , \pi ,\varepsilon }$ is non-empty, then
for $n$ large, the cardinal of $\fcal{W}_{\Upsilon , \pi ,\varepsilon }\cap \fcal{W}_{n}$ is the number of $s$-tuples
$(k_{1},\dots, k_{s}) \in \mathbb{N}^{s}$ verifying the following conditions
\begin{itemize}
\item[(K1)] $k_{1}+\cdots + k_{s} = n- \lambda_{\Upsilon, \pi, \varepsilon }$, and 
$k_{i} \geqslant 2\ell (w_{i-1})+1$ for all $1\leqslant i \leqslant s$.
\item[(K2)] For $1\leqslant i \leqslant s$, we have $k_{i}$ is even if  $\varepsilon_{i}=+$, and $k_{i}$ is odd if
$\varepsilon_{i}=-$. 
\end{itemize}

The parity condition (K2) does not allow us to conclude. We shall introduce a group action on $\Xi$ in order to
gather together different parity conditions so that only condition (K1) will be required.

\paragraph{A group action on $\Xi$.}
For $i=0,\dots ,s$, we define an application $\gamma_{i} : \Xi \rightarrow \Xi$ as follows :
for $(\Upsilon , \pi , \varepsilon )\in \Xi$ with 
$\Upsilon = (w_{0}, \dots ,w_{s})$, $\pi = (\pi_{1},\dots ,\pi_{s})$ and $\varepsilon = (\varepsilon_{1},\dots ,\varepsilon_{s})$ 
we set
$$
\gamma_{i} (\Upsilon , \pi , \varepsilon ) = ( \Upsilon' , \pi' , \varepsilon' )
$$
where
$$
\begin{array}{c}
\Upsilon'= (w_{0},\dots ,w_{i-1}, \overline{w_{i}} , w_{i+1},\dots , w_{s}) \ , \
\pi'=(\pi_{1},\dots , \pi_{i-1}, -\pi_{i}, -\pi_{i+1}, \pi_{i+2}, \dots ,\pi_{s}) \ , \\
\hbox{and } \ 
\varepsilon' = (\varepsilon_{1}, \dots ,\varepsilon_{i-1} ,-\varepsilon_{i}, \varepsilon_{i+1}, \dots ,\varepsilon_{s} ).
\end{array}
$$
Clearly, $\gamma_{i}$ is bijective, and we have $\gamma_{i}^{2} = \Id_{\Xi}$, 
$\gamma_{i}\circ \gamma_{j} = \gamma_{j}\circ \gamma_{i}$. Thus the group $\Gamma$
generated by $\gamma_{0},\dots ,\gamma_{s}$ is isomorphic to $(\mathbb{Z}/2\mathbb{Z})^{s+1}$. 
Moreover, the action of $\Gamma$ on $\Xi$ is clearly free.

Let $(\Upsilon , \pi , \varepsilon )\in \Xi$ with 
$\Upsilon = (w_{0}, \dots ,w_{s})$, $\pi = (\pi_{1},\dots ,\pi_{s})$ and $\varepsilon = (\varepsilon_{1},\dots ,\varepsilon_{s})$ 
be such that $\fcal{W}_{\Upsilon , \pi , \varepsilon}\neq \emptyset$. Then there exists
$\kappa = (k_{1},\dots ,k_{s}) \in \mathbb{N}^{s}$ verifying the conditions (K1) and (K2) above such that 
$$
w= w_{0}\zeta_{k_{1}}^{\varepsilon_{1}} w_{1} \cdots  w_{s-1} \zeta_{k_{s}}^{\varepsilon_{s}} w_{s} \in 
\fcal{W}_{\Upsilon , \pi , \varepsilon}.
$$
If $( \Upsilon' , \pi' , \varepsilon' ) = \gamma_{i} (\Upsilon , \pi , \varepsilon )$, then it is a straightforward check that 
$$
w' = w_{0}\zeta_{k_{1}}^{\varepsilon_{1}} w_{1} \cdots w_{i-1} \zeta_{k_{i}+1}^{-\varepsilon_{i}} \overline{w_{i}} 
\zeta_{k_{i+1}+1}^{\varepsilon_{i+1}} w_{i+1} \cdots w_{s-1} \zeta_{k_{s}}^{\varepsilon_{s}} w_{s} \in 
\fcal{W}_{\Upsilon' , \pi', \varepsilon'}
\ \hbox{ and } \ 
\lambda_{\Upsilon, \pi, \varepsilon} = \lambda_{\Upsilon', \pi', \varepsilon'}.
$$
It follows that if $\Omega$ denotes the $\Gamma$-orbit of $(\Upsilon, \pi , \varepsilon )$ in $\Xi$, then 
$\fcal{W}_{\Upsilon', \pi', \varepsilon'}$ is non empty for any $(\Upsilon', \pi', \varepsilon')\in \Omega$.
Moreover, we may write $\lambda_{\Omega}$ instead of $\lambda_{\Upsilon, \pi, \varepsilon}$.

Let $\Omega$ be a $\Gamma$-orbit in $\Xi$ and denote by 
$$
\fcal{W}_{\Omega} = \bigcup_{(\Upsilon , \pi, \varepsilon  ) \in \Omega} \fcal{W}_{\Upsilon , \pi, \varepsilon}.
$$
From the definition of the action of $\Gamma$, there is a unique pair $(\pi^{\circ} ,\varepsilon^{\circ} )\in \{ + , - \}^{s} \times \{ + , - \}^{s}$
such that $(\Upsilon^{\circ}, \pi^{\circ}, \varepsilon^{\circ} )\in \Omega$.

We shall write $\pi^{\circ} = (\pi^{\circ}_{1}, \dots ,\pi^{\circ}_{s})$ and 
$\varepsilon^{\circ} = ( \varepsilon^{\circ}_{1}, \dots ,\varepsilon^{\circ}_{s})$. Recall that $\Upsilon^{\circ} = (v_{0},\dots ,v_{s})$.

\paragraph{Counting elements of $\fcal{W}_{\Omega}\cap \fcal{W}_{n}$.}
Assume that $\fcal{W}_{\Omega}$ is non empty, then $\fcal{W}_{\Upsilon^{\circ}, \pi^{\circ}, \varepsilon^{\circ}}$ is non empty.
Let $n\in\mathbb{N}^{*}$ and
$$
\mathcal{K}_{n} = \bigl\{ (k_{1},\dots ,k_{s}) \in \mathbb{N}^{s} \, ;\,
k_{1}+\cdots + k_{s} = n - \lambda_{\Omega}\, , \hbox{ and 
$k_{i} \geqslant 2\ell (v_{i-1}) + 1$ for all $1\leqslant i \leqslant s$} \bigr\}.
$$

For $\kappa =  (k_{1},\dots ,k_{s}) \in \mathcal{K}_{n}$, we set $\pi_{\kappa}= (\pi_{1},\dots ,\pi_{s}) \in \{ +,-\}^{s}$ 
where 
$$
\pi_{i}  = \left\{  
\begin{array}{ll}
+ & \hbox{if $k_{i}$ is even,}\\
- & \hbox{if $k_{i}$ is odd.}
\end{array}
\right.
$$
Next, we set $\Upsilon_{\kappa} = ( w_{0}, \dots ,w_{s} ) \in [\Upsilon^{\circ}]$ where 
$$
w_{i} = \left\{ 
\begin{array}{ll}
v_{i} & \hbox{if $\sharp \{ j \ ; \ j > i$ and $\pi_{j}=\pi_{j}^{\circ}\}$ is even},\\
\overline{v_{i}} & \hbox{if $\sharp \{ j\ ;\  j > i$ and $\pi_{j}=\pi_{j}^{\circ}\}$ is odd}.
\end{array}
\right.
$$
Finally, we set $\varepsilon_{\kappa} = (\varepsilon_{1},\dots ,\varepsilon_{s}) \in \{+,-\}^{s}$ where
$$
\varepsilon_{i}  = \left\{  
\begin{array}{rl}
\varepsilon_{i}^{\circ} & \hbox{if $w_{i}=v_{i}$,}\\
-\varepsilon_{i}^{\circ} & \hbox{if $w_{i}=\overline{v_{i}}$.}
\end{array}
\right.
$$
Then one checks immediately that $(\Upsilon_{\kappa} , \pi_{\kappa} , \varepsilon_{\kappa} ) \in \Omega$ and 
$\Upsilon_{\kappa} , \varepsilon_{\kappa}$  and $\kappa$ define an element 
$w_{\kappa}$ in $\fcal{W}_{\Upsilon_{\kappa},\pi_{\kappa},\varepsilon_{\kappa}} \cap \fcal{W}_{n}$
as explained in the paragraph on the elements of $\fcal{W}_{[\Upsilon^{\circ}]}$. More precisely,
$$
w_{\kappa} = w_{0} \zeta_{k_{1}}^{\varepsilon_{1}} w_{1} \cdots w_{s-1} \zeta_{k_{s}}^{\varepsilon_{s}} w_{s}.
$$
Note that $w_{s}=v_{s}$, and so $(\Upsilon_{\kappa} , \pi_{\kappa} , \varepsilon_{\kappa} )$ is in orbit of 
$(\Upsilon^{\circ}, \pi^{\circ}, \varepsilon^{\circ} )$ 
relative to the subgroup generated by $\gamma_{0},\dots ,\gamma_{s-1}$.

To reach all the elements of $\Omega$, we observe that 
$(\gamma_{0}\cdots \gamma_{s})(\Upsilon_{\kappa}, \pi_{\kappa} , \varepsilon_{\kappa} ) = 
(\overline{\Upsilon_{\kappa}} , \pi_{\kappa} , \overline{\varepsilon_{\kappa}} )$
where $\overline{\Upsilon_{\kappa}} = (\overline{w_{0}}, \dots , \overline{w_{s}})$ and 
$\overline{\varepsilon_{\kappa}} = (-\varepsilon_{1}, \dots ,-\varepsilon_{s})$. Hence
$\overline{w_{\kappa}} = \rho w_{\kappa}\rho \in \fcal{W}_{\overline{\Upsilon_{\kappa}} , \pi_{\kappa} , 
\overline{\varepsilon_{\kappa}}}$.
We may therefore define a map 
$$
\Theta_{(\Upsilon^{\circ},\pi^{\circ}, \varepsilon^{\circ})} : 
\mathcal{K}_{n} \times \{ \iota ,\rho \} \longrightarrow
\fcal{W}_{\Omega} \cap \fcal{W}_{n} \ , \
(\kappa , \eta ) \mapsto \eta w_{\kappa} \eta \ ,
$$
which is clearly injective by definition. 

Now let $(\Upsilon , \pi , \varepsilon)\in \Omega$ with
$\Upsilon = (w_{0}, \dots ,w_{s})$, $\pi = (\pi_{1},\dots ,\pi_{s})$ and $\varepsilon = (\varepsilon_{1},\dots ,\varepsilon_{s})$,
and $w\in \fcal{W}_{\Upsilon , \pi , \varepsilon}\cap \fcal{W}_{n}$. As explained in the paragraph on the elements of 
$\fcal{W}_{[\Upsilon^{\circ}]}$, there exists $\kappa(w)= (k_{1},\dots, k_{s})\in \mathbb{N}^{s}$ verifying conditions
(K1) and (K2), such that
$$
w = w_{0} \zeta_{k_{1}}^{\varepsilon_{1}} w_{1} \cdots w_{s-1} \zeta_{k_{s}}^{\varepsilon_{s}} w_{s}.
$$
If $w_{s}=v_{s}$, then one verifies easily that $\Upsilon_{\kappa(w)} = \Upsilon$, $\pi_{\kappa (w)} = \pi$ and 
$\varepsilon_{\kappa(w)} = \varepsilon$. Hence $w = w_{\kappa(w)}$. 

If $w_{s}=\overline{v_{s}}$, then we may apply the same arguments to $\overline{w}$ and deduce that 
$w = \overline{w_{\kappa(w)}}$.

We have therefore proved that $\Theta_{(\Upsilon^{\circ},\pi^{\circ}, \varepsilon^{\circ})}$ is 
a bijection, and hence 
$$
\sharp \fcal{W}_{\Omega} \cap \fcal{W}_{n} = 2 \sharp \mathcal{K}_{n}
$$
It follows that for $n$ large, $\sharp \fcal{W}_{\Omega} \cap \fcal{W}_{n}$ is polynomial $P_{\Omega}$ in $n$ of degree $s-1$ 
with rational coefficients and strictly positive dominant coefficient.

\paragraph{A specific $\Upsilon$.}
Recall from \eqref{sbound} that $s - 1 \leqslant \left[\frac{t}{2}\right]$. To finish the proof, we are left
to produce an $\Upsilon$ such that $s - 1 = \left[\frac{t}{2}\right]$, and $\fcal{W}_{[\Upsilon]}$ is non empty.

Let $v = \tau_{0}^{-} \sigma_{0}^{+}$. A direct computation gives 
$v(\mathbfsc{i}) = \bigl( (1,2), (3) \bigr)$, and  
$$
v \bigl( ( 2 ,a_{2}^{+}, \dots ,a_{p(\mathbf{a}^{+})}^{+} ) ,  (1, a_{2}^{-} , \dots ,a_{p(\mathbf{a}^{-})}^{-} ) \bigr)
=
\bigl(
(1,4,  a_{2}^{+}, \dots ,a_{p(\mathbf{a}^{+})}^{+} ) , 
( 4 , a_{2}^{-} , \dots ,a_{p(\mathbf{a}^{-})}^{-} )
\bigr).
$$
We verify easily that $v\in  \fcal{R}_{t}$ for any $t$.

Suppose that $t$ is even. Set
$$
\Upsilon = ( \iota \, , \underbrace{\, v \, , \, \dots \, , \, v\, }_{\left[\frac{t}{2}\right] \text{ times}} , \, \iota ).
$$
Then
$$
\iota \zeta_{n - m}^{+} v \zeta_{5}^{+} v \cdots v \zeta_{5}^{+} v \zeta_{5}^{+} \iota \in \fcal{W}_{n,\Upsilon} 
$$
for $n > m = 4t +1$.

Suppose that $t$ is odd. Set
$$
\Upsilon = ( \iota \, , \underbrace{\, v \, , \, \dots \, , \, v\, }_{\left[\frac{t}{2}\right] \text{ times}} , \, v ).
$$
Then
$$
\iota \zeta_{n - m}^{+} v \zeta_{5}^{+} v \cdots v \zeta_{5}^{+} v \zeta_{5}^{+} v \in \fcal{W}_{n,\Upsilon} 
$$
for $n > m = 4t -1$.

This completes the proof of the theorem.
\end{proof}

For a fixed $t$, it is possible to give the polynomial $P_{t}$ explicitely by determining 
the set of all possible $\Upsilon$-sequences, and check 
whether $\fcal{W}_{[\Upsilon]}$ is empty or not. However, this becomes complicated
when $t$ is large. For small values of $t$, we have
$$
P_{0} = 2 \ , \ P_{1} = 8 \ , \ P_{2} = 2T+20 \ , \ P_{3} = 12T+4 \ , \ P_{4} = T^{2} + 33 T - 138.
$$

\begin{remark}
If $n\mathbfsc{a}$ denotes the bicomposition obtained from $\mathbfsc{a}$ 
by multiplying all the entries by $n$, then the exact same arguments can be applied to study 
$\fcal{M}_{n\mathbfsc{i}}$ for any $n\in\mathbb{N}^{*}$ because $n w(\mathbfsc{i}) = w (n\mathbfsc{i})$
for any $w\in \fcal{M}$. 
\end{remark}

\section{Frobenius parabolic subalgebras in $\mathrm{sl}_{n}(\Bbbk )$}

We shall establish an analogue of Theorem \ref{growth} for parabolic subalgebras using
the same method. The proofs are basically the same, but we need to treat compositions of even numbers
and odd numbers separately. This will become clear once we have the definitions of the operators analogue 
to $\sigma^{\pm}_{m}$ and $\tau_{m}^{\pm}$.
 
Recall that 
$$
\mathcal{C} = \bigcup_{n\in\mathbb{N}^{*}} \mathcal{C}_{n}.
$$

\begin{remark}\label{pmeander}
As mentioned in the introduction, the index of $\mathfrak{p}_{\mathbf{a}}$ can be computed 
via the corresponding meander graph. It follows easily from a meander graph observation that 
if $\mathfrak{p}_{\mathbf{a}}$ is Frobenius, then $a_{1} = a_{r}$ if and only if $\mathbf{a}=(1)$ or $(1,1)$.
Again, this can be recovered from \cite[Proposition 4.1 and Theorem 4.2]{Pan} or \cite{TY, TYbook}.
\end{remark}

For $m\in\mathbb{N}$, we define for $\mathbf{a}=(a_{1},\dots ,a_{r}) \in \mathcal{C}$,
$$
\theta (\mathbf{a}) = (a_{r},\dots ,a_{1}) \ , \
\sigma_{m}(\mathbf{a}) = \bigl( (m+2)a_{1} , a_{2} , \dots ,a_{r}, (m+1)a_{1} \bigr)
$$
and 
$$
\tau_{m} (\mathbf{a} ) = \left\{ 
\begin{array}{ll}
\bigl( (m+1)a_{1} + (m+2)a_{2} , a_{3} , \dots ,a_{r}, ma_{1} + (m+1)a_{2} \bigr) & \hbox{if } r > 1, \\
\mathbf{a} & \hbox{if } r=1.
\end{array}
\right.
$$
Finally, we define the compositions 
$$
\tildesigma_{m} = \theta \sigma_{m} \ \hbox{ and } \ 
\tildetau_{m} = \theta \tau_{m}.
$$

From the definitions, we have for $m\in\mathbb{N}$,
\begin{equation}\label{sigma}
s\bigl( \sigma_{m} (\mathbf{a}) \bigr) = s (\mathbf{a}) + 2(m+1)a_{1} = s \bigl( \tildesigma_{m}(\mathbf{a}) \bigr) \  , \
p \bigl( \sigma_{m} (\mathbf{a}) \bigr)  = p (\mathbf{a}) + 1 = p \bigl( \tildesigma_{m} (\mathbf{a}) \bigr) 
\end{equation}
and
\begin{equation}\label{tau}
\begin{array}{l}
s\bigl( \tau_{m} (\mathbf{a}) \bigr) =  s \bigl( \tildetau_{m} (\mathbf{a}) \bigr) =
\left\{
\begin{array}{ll} 
s (\mathbf{a}) + 2ma_{1} + 2 (m+1)a_{2} & \hbox{if } p(\mathbf{a})>1, \\
s(\mathbf{a}) & \hbox{if } p(\mathbf{a})=1,
\end{array}
\right. 
\\ [1.2em]
p \bigl( \tau_{m} (\mathbf{a}) \bigr) = p (\mathbf{a}) = p \bigl( \tildetau_{m} (\mathbf{a}) \bigr).
\end{array}
\end{equation}

Let us denote by $\pS$ the set of $\sigma_{m}$ and $\tildesigma_{m}$ where $m\in\mathbb{N}$, 
$\pT$ the set of $\tau_{m}$ and $\tildetau_{m}$ where $m\in\mathbb{N}$ and $\pA =\pS\cup\pT$.
We set $\pM$ to be the submonoid generated by $\pA$ in the set of maps from
$\mathcal{C}$ to itself. Denote by $\mathbf{1}$ the identity element of $\pM$.

\begin{remark}\label{parity}
It follows from \eqref{sigma} and \eqref{tau} that for all $w\in \pM$ and $\mathbf{a}\in \mathcal{C}$, 
we have
$$
s\bigl( w(\mathbf{a}) \bigr) - s(\mathbf{a}) \in 2\mathbb{N}.
$$
Moreover, we have $s\bigl( w(\mathbf{a}) \bigr) = s(\mathbf{a})$ if and only if $p(\mathbf{a})=1$ and
$w$ is in the submonoid generated by $\pT$. 

This explains why we need to treat compositions of even integers and odd integers separately.
\end{remark}

\begin{proposition}\label{pmonoid}
\begin{itemize}
\item[a)] For any $w\in\pM$ and $\mathbf{a}\in\mathcal{C}$, 
the index of the standard parabolic subalgebras associated to $\mathbf{a}$ and $w(\mathbf{a})$ are the same.
\item[b)] The monoid $\pM$ is free in the alphabet $\pA$.
\end{itemize}
\end{proposition}
\begin{proof}
First of all, for $\mathbfsc{a} = ( \mathbf{a}^{+},\mathbf{a}^{-} )  \in \tildeC_{n}$, denote by 
$\Theta (\mathbfsc{a}) = \bigl( \theta (\mathbf{a}^{+}) , \theta (\mathbf{a}^{-}) \bigr)$. Then 
the seaweed subalgebras $\mathfrak{p}_{\mathbfsc{a}}$ and $\mathfrak{p}_{\Theta (\mathbfsc{a})}$ are
isomorphic  \cite[Proposition 3.2]{Pan}. In particular, $\Theta$ is an index preserving operator. 

Consequently, part a) follows since for $\mathbf{a}\in \mathcal{C}_{n}$, we have the equalities
$$
\begin{array}{c}
\Theta \bigl( \mathbf{a}, (n) \bigr) = \bigl( \theta (\mathbf{a}) , (n) \bigr)
\ , \
\bigl( \sigma_{m} (\mathbf{a}) , (n+2(m+1)a_{1}) \bigr) = ( \Theta \tau_{0}^{-} \Theta \sigma_{m}^{+} ) \bigl( \mathbf{a} , (n) \bigr)
\\ [0.5em]
\hbox{and } \ 
\bigl( \sigma_{m} (\mathbf{a}) , (n+ 2ma_{1}+ 2(m+1)a_{2}) \bigr) = 
( \Theta \tau_{0}^{-} \Theta \tau_{m}^{+} ) \bigl( \mathbf{a} , (n) \bigr).
\end{array}
$$

The proof of part b) follows the same line of arguments as in the proof of Lemma \ref{uniqueness} and 
Proposition \ref{freemonoid}.
\end{proof}

Let 
$$
\mathcal{F}_{n} = \{ \, \mathbf{a}\in\mathcal{C}_{n} \, ; \, \mathrm{ind}(\mathbf{p}_{\mathbf{a}} )=0\, \}
= \{ \, \mathbf{a}\in\mathcal{C}_{n} \, ; \, \bigl( \mathbf{a}, (n) \bigr) \in \tildeF_{n} \, \}
$$ 
be the set of compositions corresponding to Frobenius standard parabolic subalgebras of $\mathrm{sl}_{n}(\Bbbk)$.
Observe from \eqref{sigma} and \eqref{tau} that the operators send compositions of even (resp. odd) 
integers to compositions of even (resp. odd) integers. For $\varepsilon \in \{ 0 , 1\}$, set
$$
\mathcal{F}^{(\varepsilon)}=\bigcup_{n\in\mathbb{N}^{*}} \mathcal{F}_{2n+\varepsilon}.
$$
Then by Proposition \ref{pmonoid}, these subsets of $\mathcal{F}$ are $\pM$-stable. Note that we have deliberately
left out $\mathcal{F}_{1}$, so we have
$$
\mathcal{F} = \mathcal{F}_{1} \cup \mathcal{F}^{(0)} \cup \mathcal{F}^{(1)}.
$$

\begin{theorem}\label{pgeneration}
Let $\mathbf{a}^{(0)}=(1,1)$, $\mathbf{a}^{(1)}=(1)$, $\pM^{(0)}=\pM$ and $\pM^{(1)}$ 
denotes the set of elements of $\pM$ which ends with a letter
in $\pS$. Then for any $\varepsilon \in \{ 0, 1\}$, the map
$$
\pM^{(\varepsilon)} \longrightarrow \mathcal{F}^{(\varepsilon)} \ , \ w \mapsto w (\mathbf{a}^{(\varepsilon)} )
$$
is bijective.
\end{theorem}
\begin{proof}
In view of Proposition \ref{pmonoid}, the proof follows the same line of arguments as in the proof of 
Theorem \ref{frobeniusbijection}.
\end{proof}

For $p,n\in\mathbb{N}^{*}$, we denote
$$
\mathcal{F}_{n,p} = \{ \mathbf{a}\in\mathcal{F}_{n} \, ; \, p(\mathbf{a})=p\}.
$$

\begin{theorem}\label{pgrowth}
\begin{itemize}
\item[a)] Let $p,n\in\mathbb{N}^{*}$. Then $\mathcal{F}_{n,p}$ is empty if $p > \left[ \frac{n}{2} \right] +1$.
\item[b)] Let $\varepsilon \in \{ 0,1\}$ and $t\in\mathbb{N}$. There exists a polynomial
$P_{\varepsilon , t}\in\mathbb{Q}[T]$ of degree $\left[\frac{t}{2}\right]$ and with positive dominant coefficient
such that 
$$
\sharp \mathcal{F}_{2n+\varepsilon , n+1-t} = P_{\varepsilon,t} (n)
$$
for large $n$.
\end{itemize}
\end{theorem}

Let $w\in\pM$ and $\mathbf{a}\in \mathcal{C}$. Then we may define $\ell (w), \sigma( w), \tau (w)$ and
the $w$-sequence $\mathbf{m}$ of $\mathbf{a}$ in the same manner. Since the $w$-sequence of $\mathbf{a}$ 
consists only of even integers, we set $\beta (\mathbf{m})=\{ i \, ; \, m_{i} > 2\}$.

We still have the equalities
\begin{equation}\label{pOb3}
\ell (w) = \sigma (w)  + \tau (w) \ , \ 
p\bigl( w(\mathbf{a}) \bigr) = p(\mathbf{a}) + \sigma (w).
\end{equation}
But (Ob4) becomes 
\begin{equation}\label{pOb4}
2 \ell (w) + 2 \beta (w) \leqslant m_{1}+\cdots + m_{k} = s \bigl( w(\mathbf{a}) \bigr) - s(\mathbf{a}).
\end{equation}
because the $w$-sequence of $\mathbf{a}$ consists of even integers.

We have the following properties of the operators in $\pA$ relative to $w$-sequences which are analogues
of Proposition \ref{identical}, Lemmas \ref{(1)-step}, \ref{(1,1,1)-step} and Remark \ref{zetadescription}:
\begin{proposition}\label{pidentical}
Let $w\in\pM$, and $\mathbf{a},\mathbf{b}\in \mathcal{C}$ be such that
\begin{itemize}
\item[i)] $p(\mathbf{a}) \geqslant 2\ell (w)+1$, $p(\mathbf{b})\geqslant 2\ell (w)+1$, 
\item[ii)] $a_{i}=b_{i}$ for all $1\leqslant i\leqslant \ell (w)+1$, and 
\item[iii)] $a_{p(\mathbf{a})-i+1} = b_{p(\mathbf{b})-i+1}$ for all $1\leqslant i\leqslant \ell( w)$.
\end{itemize} 
Then the $w$-sequence of $\mathbf{a}$ and the $w$-sequence of $\mathbf{b}$ are identical. It is
completely determined by $a_{1},\dots ,a_{\ell (w)+1}, a_{p(\mathbf{a})-\ell( w) +1}, \dots ,a_{p(\mathbf{a})}$.
In particular, we have
$$
s\bigl( w(\mathbf{a}) \bigr) - s(\mathbf{a}) = s\bigl( w(\mathbf{b}) \bigr) - s(\mathbf{b}).
$$
\end{proposition}

\begin{lemma}\label{(2)-step}
Let $w\in\pM$, $\eta \in \pA$, $\mathbf{a}\in\mathcal{C}$ and $\mathbf{b}=\eta (\mathbf{a})$.
\begin{itemize}
\item[a)] If $s\bigl( w(\mathbf{a}) \bigr) - s(\mathbf{a}) =2$, then $w\in \pA$, and we have
\begin{itemize}
\item[i)] either $w\in \{ \sigma_{0}, \tildesigma_{0}\}$ and $a_{1}=1$, 
\item[ii)] or $p(\mathbf{a}) \geqslant 2$, $w\in  \{ \tau_{0}, \tildetau_{0}\}$ and $a_{2}=1$.
\end{itemize}
\item[b)] If $s\bigl( \eta (\mathbf{a}) \bigr) - s(\mathbf{a}) > 2$, then $b_{1} \geqslant 2$.
\item[c)] If $\mathbf{b}\neq \mathbf{a}$ and $b_{1}=1$, then 
\begin{itemize}
\item[i)] either $\eta = \tildesigma_{0}$ and $a_{1}=1$, 
\item[ii)] or $p(\mathbf{a}) \geqslant 2$, $w = \tildetau_{0}$ and $a_{2}=1$.
\end{itemize}
\end{itemize}
\end{lemma}

\begin{lemma}\label{(2,2,2)-step}
Let $\eta,\eta'\in \pA$ and  $\mathbf{a}\in \mathcal{C}$.
\begin{itemize}
\item[a)] If  the $\eta \tildesigma_{0}^{2}$-sequence of $\mathbf{a}$ is $(2,2,2)$, then
$\eta \in\{\sigma_{0} , \tildesigma_{0}\}$.
\item[b)] If the $\eta'\eta \tildesigma_{0}^{2}$-sequence of $\mathbf{a}$ is $(2,2,2,2)$, then 
$\eta= \tildesigma_{0}$ and $\eta' \in\{\sigma_{0} , \tildesigma_{0}\}$.
\end{itemize}
\end{lemma}

\begin{remark}\label{tildesigma0description}
Let $\mathbf{a}\in \mathcal{C}$ be such that $a_{1}=1$. Then for $m\in \mathbb{N}$, we have
$$
\tildesigma_{0}^{2m} (\mathbf{a}) = 
( 1,\underbrace{2,\dots ,2}_{m \text{ times}}, a_{2},\dots ,a_{p(\mathbf{a})} , \underbrace{2,\dots ,2}_{m \text{ times}})
$$
and 
$$
\tildesigma_{0}^{2m+1} (\mathbf{a}) = 
( 1,\underbrace{2,\dots ,2}_{m \text{ times}}, a_{p(\mathbf{a})},\dots ,a_{2} , \underbrace{2,\dots ,2}_{m+1 \text{ times}}).
$$
\end{remark}

\begin{proofof}{Proof of Theorem \ref{pgrowth}}
a) This is proved in the same manner as Proposition \ref{empty} by using Theorem \ref{pgeneration} and the identities
in \eqref{pOb3} and \eqref{pOb4}. Of course, we have to treat the cases $\varepsilon =0$ and $\varepsilon =1$ 
separately.

b) We shall proceed as in the proof of Theorem \ref{growth}. Let us fix $\varepsilon \in \{ 0 ,1 \}$ and 
$n\in\mathbb{N}^{*}$ to be large enough. By Theorem \ref{pgeneration}, $\mathcal{F}_{2n +\varepsilon , n+1-t}$
is in bijection with
$$
\pfont{W}^{(\varepsilon)}=\{ w\in \pM^{(\varepsilon)}  \, ; \, w(\mathbf{a}^{(\varepsilon)}) \in \mathcal{F}_{2n +\varepsilon , n+1-t}\}.
$$
Let $w=\eta_{1}\cdots \eta_{k}\in \pfont{W}^{(\varepsilon)}\setminus \{\mathbf{1}\}$ where $\eta_{1},\dots ,\eta_{k}\in\pA$, and
denote by $\mathbf{m}=(m_{1},\dots ,m_{k})$ the $w$-sequence of $\mathbf{a}$. 
Observe that $\eta_{k}\in\pS$ if $\varepsilon=1$. It follows that 
$p\bigl( \eta_{k}(\mathbf{a}^{(\varepsilon)}) \bigr) >1$, which implies that 
$m_{i}\geqslant 2$ for all $i=1,\dots ,k$.

Since $p(\mathbf{a}^{(\varepsilon)})=s(\mathbf{a}^{(\varepsilon)})=2-\varepsilon$, we 
obtain by using the same arguments the following analogue of (Ob5) :
$$
n+\varepsilon - t -1 = \sigma (w) \leqslant \ell( w ) \leqslant \ell (w) + \beta (\mathbf{m})
\leqslant \frac{1}{2} \Bigl( s\bigl( w(\mathbf{a}^{(\varepsilon)})\bigr) - s(\mathbf{a}^{(\varepsilon)}) \Bigr)
= n+\varepsilon -1 , 
$$
and therefore,
$$ 
0 \leqslant \tau (w) + \beta (\mathbf{m}) \leqslant t.
$$
We deduce from this and the inequality
$2\bigl(\ell(w)-1\bigr)+m_{i}\leqslant 2(n+\varepsilon -1)$ that
$$
m_{i}\leqslant 2(n-\ell (w) + \varepsilon ) \leqslant 2t +2.
$$
It follows that 
$$
\eta_{i}\in \pA_{t} = \{ \sigma_{0},\dots ,\sigma_{t},\tildesigma_{0},\dots ,\tildesigma_{t} ,
\tau_{0},\dots ,\tau_{t},\tildetau_{0},\dots ,\tildetau_{t} \}
$$
for all $1\leqslant i\leqslant k$. This is the analogue of (Ob6), and we have shown that 
$w$ is an element in the submonoid $\pM_{t}$ generated by $\pA_{t}$.

\paragraph{A first decomposition of $w$.}
In this case, our first decomposition is slightly simpler. Recall that 
$$
\begin{array}{cccccccc}
w & = & \eta_{1} & \eta_{2} & \cdots & \eta_{k} & \\
\mathbf{m} & = & (m_{1})& (m_{2}) & \cdots & (m_{k}) & 
\end{array}
$$
We rewrite this decomposition by regrouping the pairs $(\tildesigma_{0},2)$. 
It follows that there exists $q\in\mathbb{N}$ and $k_{1},\dots , k_{q}\in\mathbb{N}^{*}$ such that 
$$
\begin{array}{cccccccccc}
w & = & w_{0} & \tildesigma_{0}^{k_{1}}  & w_{1} & \cdots & w_{q-1} & \tildesigma_{0}^{k_{q}} & w_{q} \\
\mathbf{m} & = & \mathbf{m}_{0}& \mathbf{n}_{1} & \mathbf{m}_{1} & \cdots & \mathbf{m}_{q-1} & 
\mathbf{n}_{q} & \mathbf{m}_{q}
\end{array}
$$
where the $w_{i}$'s does not contain any pairs of the form $(\tildesigma_{0},2)$. By construction,
$\mathbf{n}_{i} = ( \underbrace{2 ,\dots ,2}_{k_{i} \text{ times}} )$. Note that 
$w_{1},\dots ,w_{q-1} \in\pM \setminus \{\pfont{1}\}$.

This decomposition is therefore unique, and we define the $\Delta$-sequence of $w$ to be the 
$\Delta (w) = (w_{0},\dots ,w_{q})$.

\paragraph{Conditions on $q$ and $w_{i}$.}
It follows from Lemma \ref{(2)-step} part c) that for $1\leqslant i\leqslant q-1$, we have
$w_{i} = \tildetau_{0} w_{i}'$ for some $w'\in\pM$.  Thus
$$
q-1\leqslant \tau (w) \leqslant t.
$$

Now let $0\leqslant i\leqslant q$. There exist $w', w'' \in \pM$ such that $w = w'w_{i}w''$. 
Thus $\mathbf{m}_{i}$ is the $w_{i}$-sequence of $w''(\mathbf{a}^{(\varepsilon)})$.
Let $d,h\in\mathbb{N}$ be such that $w_{i}=\eta_{d} \cdots \eta_{d+h}$. Then
$\mathbf{m}_{i}=(m_{d},\dots ,m_{d+h})$. By the construction of the first decomposition, 
we have $(\eta_{d+j}, m_{d+j}) \neq (\tildesigma_{0}, 2)$ for $0\leqslant j\leqslant h$. So 
for $0\leqslant j\leqslant h$, we have
either $\eta_{d+j} \in \pT$, 
$m_{d+j} > 2$,
or $(\eta_{d+j}, m_{d+j}) = (\sigma_{0}, 2)$.

If $(\eta_{d+j}, m_{d+j}) = (\sigma_{0}, 2)$, then by Lemma \ref{(2)-step}, we have either
$j=h$ or $\eta_{d+j+1} \in \pT$. So the number of $j$'s such that $(\eta_{d+j}, m_{d+j}) = (\sigma_{0}, 2)$
is at most $\tau (w_{i}) +1$.

Consequently, we have
$$
\ell (w_{i} ) \leqslant 2 \tau (w_{i}) +Ê\beta (\mathbf{m}_{i}) +1 \leqslant 2t + 1.
$$
We deduce that 
$$
w_{i}\in \pfont{N}_{t} = \{ u\in \pM_{t}  \, ; \, \ell (u) \leqslant 2t+1 \}
$$
which is a finite set depending only on $t$, and so 
$$
\pfont{W}^{(\varepsilon)} = \bigcup_{q=0}^{t+1} \bigcup_{\Delta\in \pfont{N}_{t}^{q+1}} \pfont{W}^{(\varepsilon)}_{\Delta}
$$
where $\pfont{W}^{(\varepsilon)}_{\Delta}= \{ u\in \pfont{W}^{(\varepsilon)} \, ; \, \Delta (u) =\Delta\}$.

\paragraph{A second decomposition of $w$.}
We define a second decomposition of $w$ as in the proof of Theorem \ref{growth} by replacing 
$\fcal{M}$ by $\pM$, and  $\fcal{Z}$ by
$$
\pfont{Z} = \{ \tildesigma_{0}^{h} \, ; \, h\in\mathbb{N}^{*}\}.
$$
The condition $(*)$ is unchanged, and we define $\pfont{R}_{t}$ and  $\Upsilon (w)$ 
in the same manner.  The set $\pfont{R}_{t}$ is finite and its cardinal depends only on $t$.  

We have therefore the second decomposition 
$$
\begin{array}{cccccccccc}
w & = & v_{0} & u_{1} & v_{1} & \cdots & v_{s-1} & u_{s} & v_{s} \\
\mathbf{m} & = & \mathbf{m}_{0}' & \mathbf{n}_{1}' & \mathbf{m}_{1}' & \cdots &
\mathbf{m}_{s-1}' & \mathbf{n}_{s}' & \mathbf{m}_{s}' 
\end{array}
$$
with $v_{0},\dots ,v_{s} \in \pfont{R}_{t}$ and $u_{1},\dots ,u_{s}\in \pfont{Z}$ 
verifying
$$
\ell (u_{i}) \geqslant 2\ell (v_{i-1}) +1
$$
for $i=1,\dots ,s$. Observe that we have $\mathbf{n}_{i}'=(2,\dots ,2)$ for $1\leqslant i\leqslant s$.

For $2\leqslant i\leqslant s$, we have 
$\ell (u_{i}) \geqslant  2\ell (v_{i-1}) +1 \geqslant 3$ because $v_{i-1} \neq \pfont{1}$.
It follows from Lemma \ref{(2,2,2)-step} that $\beta (m_{i-1}) > 0$.
Hence 
$$
s-1 \leqslant \beta (\mathbf{m}),
$$ 
and therefore together with the fact that $s-1 \leqslant q-1 \leqslant \tau (w)$, we obtain 
$$
s \leqslant \left[\frac{t}{2}\right] +1.
$$
Thus we have 
$$
\pfont{W}^{(\varepsilon)} =
\bigcup_{s=0}^{\left[\frac{t}{2}\right] +1} \bigcup_{\Upsilon \in \pfont{R}_{t}^{s+1}} 
\pfont{W}^{(\varepsilon)}_{\Upsilon}
$$
where $\pfont{W}^{(\varepsilon)}_{\Upsilon} = \{ u\in \pfont{W}^{(\varepsilon)} \, ; \, \Upsilon (u) =\Upsilon\}$.

Using Proposition \ref{pidentical} and Remark \ref{tildesigma0description}, 
we may apply the same arguments as in the proof of Theorem \ref{growth} to obtain
that for $\Upsilon \in \pfont{R}_{t}^{s+1}$ and $n$ large, 
either $\pfont{W}^{(\varepsilon)}_{\Upsilon}$ is empty or 
$\sharp \pfont{W}^{(\varepsilon)}_{\Upsilon}$ is given by a polynomial
$P_{\Upsilon}^{(\varepsilon)}$ of degree $s-1$ with rational coefficiants and strictly positive
dominant coefficient.

Note that in this case, we need to treat the cases $\varepsilon=0$ and $\varepsilon =1$ separately,
but we do not need to keep track of the parities of $\ell (u_{i})$ anymore, thus making the proof much 
simpler.

\paragraph{Final step.}
So to finish the proof, we are left to provide an $\Upsilon$ such that $s-1=  \left[\frac{t}{2}\right]$
and $\pfont{W}^{(\varepsilon)}_{\Upsilon}$ is non empty.

Let $v= \tildetau_{0}\tildesigma_{0}\sigma_{0}$. A direct computation gives
$$
v(1,a_{2},\dots,a_{r}) = (1,4,a_{2},\dots ,a_{r},4).
$$

We set $\Upsilon$ according to the values of $t$ and $\varepsilon$ as follows:
$$
\begin{array}{c|c|c}
t & \varepsilon & \Upsilon \\ \hline\hline
\hbox{even} & 0 \hbox{ or } 1 &  
(\pfont{1} ,\underbrace{v ,\dots , v}_{\left[\frac{t}{2}\right] \text{ times}}, \pfont{1}) \\ \hline 
\hbox{odd} & 0 & 
(\pfont{1} ,\underbrace{v ,\dots , v}_{\left[\frac{t}{2}\right] \text{ times}}, \tildetau_{0}) \\ \hline
\hbox{odd} & 1 & 
(\pfont{1} ,\underbrace{v ,\dots , v}_{\left[\frac{t}{2}\right] \text{ times}}, \tildetau_{0}\sigma_{0} ) 
\end{array}
$$

We check as in the proof of Theorem \ref{growth} that in each one of these three cases, we have 
$\Upsilon \in \pfont{R}_{t}^{\left[\frac{t}{2}\right]+2}$, and 
$\pfont{W}^{(\varepsilon)}_{\Upsilon}$ is non empty for $n$ large. 
We have therefore completed our proof.
\end{proofof}

\begin{remark}
Observe that in the proof of Theorem \ref{pgrowth}, the condition $(*)$ 
used in the second decomposition can be optimized to  $\ell (u_{h}) < 2 \ell (w_{i}u_{1}w_{i+1}\cdots w_{h+i-1})$. 
\end{remark}

As in the case of Frobenius seaweed subalgebras, it is possible to determine 
the polynomials $P_{\varepsilon,t}$ explicitely. For small values of $t$, we have
$$
P_{\varepsilon,0} = 2 \ , \ P_{0,1} = 12 \ , \ P_{1,1} = 6 \ , \ P_{1,2} = 2T+12.
$$

{\parindent = 0cm
\small
Michel Duflo, \\ [-2.5pt]
Universit\'e Denis Diderot-Paris 7, \\[-2.5pt]
Institut de Math\'ematiques de Jussieu, \\[-2.5pt]
Case Postale 7012, \\[-2.5pt]
2 place Jussieu, \\[-2.5pt]
75251 Paris cedex 05, \\[-2.5pt]
France. \\[-2.5pt]
{\tt michel.duflo@imj-prg.fr}}
\bigskip

{\parindent = 0cm
\small
Rupert W.T. Yu, \\[-2.5pt]
Laboratoire de Math\'ematiques de Reims EA 4535, \\[-2.5pt]
U.F.R. Sciences Exactes et Naturelles, \\[-2.5pt]
Universit\'e de Reims Champagne Ardenne, \\[-2.5pt]
Moulin de la Housse - BP 1039, \\[-2.5pt]
51687 Reims cedex 2, \\[-2.5pt]
France. \\[-2.5pt]
{\tt rupert.yu@univ-reims.fr}
}

\end{document}